\documentclass[11pt,a4paper]{amsart}
\usepackage{amsfonts,amsmath,amssymb}
\usepackage{hyperref}
\usepackage{latexsym,fullpage,amsfonts,amssymb,amsmath,amscd,graphics,epic}
\usepackage[all]{xy}
\usepackage{amssymb,amsthm,amsxtra}
\usepackage[usenames]{color}
\usepackage{amscd}
\usepackage{amsthm}
\usepackage{amsfonts}
\usepackage{amssymb}
\usepackage{mathrsfs}
\usepackage{url}
\usepackage{bbm}
\usepackage{wasysym}
\usepackage{enumitem}
\usepackage[vcentermath]{youngtab}%
\theoremstyle{plain}
\newtheorem*{theorem*}{Theorem}
\newtheorem*{remark*}{Remark}
\newtheorem*{example*}{Example}
\newtheorem{lemma}{Lemma}[subsection]
\newtheorem{proposition}[lemma]{Proposition}
\newtheorem{corollary}[lemma]{Corollary}
\newtheorem{theorem}[lemma]{Theorem}

\newtheorem*{conjecture*}{Conjecture}

\newtheorem{sublemma}[lemma]{Sublemma}

\theoremstyle{definition}
\newtheorem{definition}[lemma]{Definition}

\theoremstyle{remark}
\newtheorem{remark}[lemma]{Remark}

\newtheorem{notation}[lemma]{Notation}

 \newcommand{\idealI}{\mathfrak{I}}
\newcommand{\Hom}{\operatorname{Hom}}

\newcommand{\triv}{{\mathbbm 1}}

\newcommand{\id}{\operatorname{Id}}

\renewcommand{\Im}{\operatorname{Im}}
\newcommand{\Ker}{\operatorname{Ker}}

\newcommand{\co}{{\it O}}

\newcommand{\End}{\operatorname{End}}

\newcommand{\bC}{{\mathbb C}}
\newcommand{\bZ}{{\mathbb Z}}

\newcommand{\Del}{{\Delta}}

\newcommand{\lam}{{\lambda}}

\newcommand{\fh}{{\mathfrak{h}}}

\newcommand{\T}{\tau}

\newcommand{\gl}{{\mathfrak{gl}}}

\newcommand{\abs}[1]{\left|{#1}\right|}

\newcommand{\Dab}{\underline{Rep}^{ab}(S_{\nu})}

\newcommand{\InnaA}[1]{#1}
\newcommand{\InnaB}[1]{#1}
\newcommand{\InnaC}[1]{#1}
\newcommand{\InnaD}[1]{#1}
\newcommand{\InnaE}[1]{#1}
\newcommand{\InnaF}[1]{#1}
\newcommand{\InnaG}[1]{#1}
\def\quotient#1#2{%
    \raise1ex\hbox{$#1$}\Big/\lower1ex\hbox{$#2$}%
}

\begin{document}

\date{\today}
\title{Schur-Weyl duality for Deligne categories II: the limit case}
 \author{Inna Entova Aizenbud}
\address{Inna Entova Aizenbud,
Massachusetts Institute of Technology,
Department of Mathematics,
Cambridge, MA 02139 USA.}
\email{inna.entova@gmail.com}

\begin{abstract}
 This paper is a continuation of \InnaG{a previous paper of the author (\cite{EA}), which gave} an analogue to the classical Schur-Weyl duality in the setting of Deligne categories. 
 
 Given a finite-dimensional unital vector space $V$ (a vector space $V$ with \InnaC{a chosen non-zero vector $\mathbbm{1}$}), we constructed in \cite{EA} a complex tensor power of $V$: \InnaG{an $Ind$-object of the Deligne category $\underline{Rep}(S_t)$ which is a Harish-Chandra module for the pair $(\gl(V), \bar{\mathfrak{P}}_{\triv})$, where
$\bar{\mathfrak{P}}_{\triv} \subset GL(V)$ is the mirabolic subgroup preserving the vector $\mathbbm{1}$.}


 This construction allowed us to obtain an exact contravariant functor $\widehat{SW}_{t, \InnaC{V}}$ from the category $\underline{Rep}^{ab}(S_t)$ (the abelian envelope of the category $\underline{Rep}(S_t)$) to a certain \InnaC{localization} of the parabolic category $\co$ associated with the pair $(\gl(V), \bar{\mathfrak{P}}_{\triv})$. 

\InnaC{In this paper, we consider the case when $V = \bC^{\infty}$. We define the appropriate version of the parabolic category $\co$ and its localization, and show that the latter is equivalent to a \InnaD{``restricted''} inverse limit of categories $\widehat{\co}^{\mathfrak{p}}_{t,\bC^N}$ with $N$ tending to infinity. The Schur-Weyl functors $\widehat{SW}_{t, \InnaC{\bC^N}}$ then give an anti-equivalence between this category and the category $\underline{Rep}^{ab}(S_t)$.}

This duality provides an unexpected tensor structure on the \InnaD{category} $\widehat{\co}^{\mathfrak{p}_{\infty}}_{t,\bC^{\infty}}$.
\end{abstract}

\keywords{Deligne categories, Schur-Weyl duality}
\maketitle
\setcounter{tocdepth}{3}
\section{Introduction}\label{sec:intro}

\subsection{}

The Karoubian rigid symmetric monoidal categories $\underline{Rep}(S_t)$, $t \in \bC$, were defined by P. Deligne in \cite{D} as a polynomial family of categories interpolating the categories of finite-dimensional representations of the symmetric groups; namely, at points $n=t \in \bZ_+$ \InnaD{the category $\underline{Rep}(S_{t=n})$ allows} an essentially surjective additive symmetric monoidal functor onto the standard \InnaD{category} $Rep(S_n)$. The \InnaD{categories $\underline{Rep}(S_t)$ were} subsequently studied by P. Deligne and others (e.g. by V. Ostrik, J. Comes in \cite{CO}, \cite{CO2}). 

In \cite{EA}, we gave an analogue to the classical Schur-Weyl duality in the setting of Deligne categories. In order to do this, we defined the ``complex tensor power'' of a finite-dimensional unital \InnaA{complex} vector space (i.e. a vector space $V$ with a distinguished non-zero vector $\InnaB{\triv}$). This ``complex tensor power'' of $V$, \InnaD{denoted by $V^{\underline{\otimes} t}$,} is an $Ind$-object in the category $\underline{Rep}(S_t)$, and comes with an action of $\gl(V)$ on it; \InnaG{moreover, this $Ind$-object is a Harish-Chandra module for the pair $(\gl(V), \bar{\mathfrak{P}}_{\triv})$, where
$\bar{\mathfrak{P}}_{\triv} \subset GL(V)$ is the mirabolic subgroup preserving the vector $\triv$.}

The ``$t$-th tensor power'' of $V$ is defined for any $t \in \bC$; for $n=t \in \bZ_+$, the functor $\underline{Rep}(S_{t=n}) \rightarrow Rep(S_n)$ takes this $Ind$-object of $\underline{Rep}(S_{t=n})$ to the usual tensor power $V^{\otimes n}$ in $Rep(S_n)$. Moreover, the action of $\gl(V)$ on the former object corresponds to the action of $\gl(V)$ on $V^{\otimes n}$.

This allowed us to define an additive contravariant functor, called the Schur-Weyl functor,
$$SW_{t, V}: \underline{Rep}^{ab}(S_{t}) \longrightarrow \co^{\mathfrak{p}}_{V}, \InnaD{\; \; SW_{t, V}:=\Hom_{\underline{Rep}^{ab}(S_{t})}( \, \cdot \, , V^{\underline{\otimes} t})}$$
Here $\underline{Rep}^{ab}(S_{t})$ is the abelian envelope of the category $\underline{Rep}(S_{t})$ (this envelope was described in \cite{CO2}, \cite[Chapter 8]{D}). \InnaE{The} category $\co^{\mathfrak{p}}_{V}$ is \InnaE{a version of} the parabolic category $\co$ for $\gl(V)$ associated with the \InnaB{pair $(V, \InnaB{ \triv})$}, \InnaE{which is defined as follows. 

We define $\co^{\mathfrak{p}}_{V}$ to be the category of Harish-Chandra modules for the pair $(\gl(V), \bar{\mathfrak{P}}_{\triv})$ on which the group $GL(V /\bC \triv)$ acts by polynomial maps, and which satisfy some additional finiteness conditions (similar to the ones in the definition of the usual BGG category $\co$).} 

\InnaD{We now consider the localization of $\co^{\mathfrak{p}}_{V}$ obtained by taking the full subcategory of $\co^{\mathfrak{p}}_{V}$ consisting of modules of \InnaA{degree} $t$ (i.e. modules on which $\id_V \in \End(V)$ acts by the scalar $t$), and localizing} by the Serre subcategory of \InnaE{$\gl(V)$-polynomial} modules. This quotient is denoted by $\widehat{\co}^{\mathfrak{p}}_{t, V}$.

It turns out that for any unital finite-dimensional space $(V, \triv)$ and any $t \in \bC$, the contravariant functor $\InnaF{\widehat{SW}}_{t, V}$ \InnaD{makes $\widehat{\co}^{\mathfrak{p}}_{t, V}$} a Serre quotient of $\underline{Rep}^{ab}(S_{t})^{op}$.

In this paper, we will consider the categories $\widehat{\co}^{\mathfrak{p}_N}_{t, \bC^N}$ for $N \in \bZ_+$ and for $N =\infty$. 

Defining appropriate restriction functors $$\widehat{\InnaD{\mathfrak{Res}}}_{n-1, n}:\widehat{\co}^{\mathfrak{p}_n}_{t, \bC^n} \longrightarrow \widehat{\co}^{\mathfrak{p}_{n-1}}_{t, \bC^{n-1}}$$
\InnaC{allows us to} consider the inverse limit of the system $((\widehat{\co}^{\mathfrak{p}_n}_{t, \bC^n})_{n \geq 0}, (\widehat{\InnaD{\mathfrak{Res}}}_{n-1, n})_{n \geq 1})$. \InnaC{Inside this inverse limit we consider a full subcategory which is equivalent to $\widehat{\co}^{\mathfrak{p}_{\infty}}_{t, \bC^{\infty}}$. This subcategory is called the ``the restricted inverse limit'' of the system $((\widehat{\co}^{\mathfrak{p}_n}_{t, \bC^n})_{n \geq 0}, (\widehat{\InnaD{\mathfrak{Res}}}_{n-1, n})_{n \geq 1})$, and will be denoted by $\varprojlim_{n \geq 1, \text{ restr}} \widehat{\co}^{\mathfrak{p}_{n}}_{t, \bC^{n}}$.} \InnaG{This category has an intrinsic description, which we give in this paper (intuitively, this is the inverse limit among finite-length categories).}


\InnaC{\InnaD{Similarly to \cite{EA}, we define} the complex tensor power of \InnaG{the unital vector space} $(\bC^{\infty}, \triv:=e_1)$, and the corresponding Schur-Weyl contravariant functor \InnaD{${SW}_{t, \bC^{\infty}}$. As in the finite-dimensional case, this functor induces an exact contravariant functor $\widehat{SW}_{t, \bC^{\infty}}$, and} we have the following commutative diagram:}

 $$\xymatrix{ &\underline{Rep}^{ab}(S_{t})^{op} \ar[rr]_{\widehat{SW}_{t, \InnaD{\text{lim}}}}  \ar[rrd]_{\widehat{SW}_{t, \bC^{\infty}}} &{} &\varprojlim_{n \geq 1, \text{ restr}} \widehat{\co}^{\mathfrak{p}_{n}}_{t, \bC^{n}}  \\ &{} &{} &\widehat{\co}^{\mathfrak{p}_{\infty}}_{t, \bC^{\infty}} \ar[u]  }$$

 \InnaD{The contravariant functors $\widehat{SW}_{t, \bC^{\infty}}$, $\widehat{SW}_{t, \InnaD{\text{lim}}}$ turn out to be} anti-equivalences
induced by the Schur-Weyl functors $SW_{t, \bC^n}$. 

The anti-equivalences $\widehat{SW}_{t, \bC^{\infty}}$, $\widehat{SW}_{t, \InnaD{\text{lim}}}$ induce an unexpected structure of a rigid symmetric monoidal category on $$\widehat{\co}^{\mathfrak{p}_{\infty}}_{t, \bC^{\infty}} \cong \varprojlim_{\InnaD{n \geq 1}, \text{ restr}} \widehat{\co}^{\mathfrak{p}_n}_{t, \bC^n}$$ \InnaD{We obtain an interesting corollary: the duality in this category given by the tensor structure will coincide with the one arising from the usual notion of duality in BGG category $\co$.}

\subsection{Notation and structure of the paper}
The base field throughout the paper will be $\bC$. \InnaG{The notation and definitions used thoughtout this paper can be found in \cite[Section 2]{EA}.}

\InnaG{Sections \ref{sec:Del_cat_S_nu} and \ref{sec:poly_rep} contain preliminaries on the Deligne category $\underline{Rep}(S_{\nu})$ (thoughout the paper, we use the parameter $\nu$ instead of $t$), the categories of polynomial representations of $\gl_N$ ($N \in \bZ_+ \cup \{ \infty \}$) and the parabolic category $\co$ for $\gl_N$. These sections are based on \cite{EA}, \cite{EA1} and \cite{SS}. }




In Section \ref{sec:lim_par_cat_O}, we give a description of the parabolic category $\co$ for $\gl_{\infty}$ as a restricted inverse limit of the parabolic categories $\co$ for $\gl_n$ as $n$ tends to infinity. 

In Sections \ref{sec:comp_tens_power} and \ref{sec:SW_functor}, we recall the definition of the complex tensor power $\InnaE{(\bC^N)^{\underline{\otimes} \nu}}$, and define the functors $SW_{\nu, V} : \underline{Rep}^{ab}(S_{\nu})^{op} \rightarrow \co^{\mathfrak{p}}_{\nu, V}$ and $\widehat{SW}_{\nu, V} : \underline{Rep}^{ab}(S_{\nu})^{op} \rightarrow \widehat{\co}^{\mathfrak{p}}_{\nu, V}$ for a unital vector space $(V, \triv)$ \InnaD{(finite or infinite-dimensional)}. \InnaD{In Subsection \ref{ssec:finite_dim_SW}, we recall the finite-dimensional case (studied in \cite{EA}).}
 

 \InnaG{Section \ref{sec:class_SW_inv_limit} discusses the restricted inverse limit construction in the case of the classical Schur-Weyl duality, which motivates our construction for the Deligne categories. 
 Sections \ref{sec:SW_duality_limit} and \ref{sec:SW_duality_inf_dim} prove the main results of the paper. 
 Section \ref{sec:duality} discusses the relation between the rigidity (duality) in $\underline{Rep}^{ab}(S_{\nu})$ and the duality in the parabolic category $\co$ for $\gl_{\infty}$.}
%


\subsection{Acknowledgements}
I would like to thank my advisor, Pavel Etingof, for his guidance and suggestions.

\section{Deligne category \texorpdfstring{$\underline{Rep}(S_{\nu})$}{}}\label{sec:Del_cat_S_nu}
\InnaG{A detailed description of the Deligne category $\underline{Rep}(S_t)$ and its abelian envelope can be found in \cite{CO, CO2, D, E1} as well as \cite{EA1}}. Throughout the paper, we will use the parameter $\nu$ instead of the parameter $t$ used in the introduction.
\subsection{General description}\label{ssec:S_nu_general}
%
For any $\nu \in \bC$, the category $\underline{Rep}(S_{\nu})$ is generated, as a $\bC$-linear Karoubian tensor category, by one object, denoted $\fh$. This object is the analogue of the permutation representation of $S_n$, and any object in $\underline{Rep}(S_{\nu})$ is a direct summand in a direct sum of tensor powers of $\fh$.

For $\nu \notin \bZ_{+}$, $\underline{Rep}(S_{\nu})$ is a semisimple abelian category.


If $\nu$ is a non-negative integer, then the category $\underline{Rep}(S_{\nu})$ has a tensor ideal $\idealI_{\nu}$, called the ideal of negligible morphisms (this is the ideal of morphisms $f: X \longrightarrow Y$ such that $tr(fu)=0$ for any morphism $u: Y \longrightarrow X$). In that case, the classical category $Rep(S_n)$ of finite-dimensional representations of the symmetric group for $n:=\nu$ is equivalent to $\underline{Rep}(S_{\nu=n})/\idealI_{\nu}$ (equivalent as Karoubian rigid symmetric monoidal categories).

The full, essentially surjective functor $\underline{Rep}(S_{\nu=n}) \rightarrow Rep(S_n)$ defining this equivalence will be denoted by $\mathcal{S}_n$.

Note that $\mathcal{S}_n$ sends $\fh$ to the permutation representation of $S_n$.

The indecomposable objects of $\underline{Rep}(S_{\nu})$, regardless of the value of $\nu$, are \InnaA{parametrized (up to isomorphism)} by all Young diagrams (of arbitrary size). We will denote the indecomposable object in $\underline{Rep}(S_{\nu})$ corresponding to the Young diagram $\T$ by $X_{\T}$.

For non-negative integer $\nu =:n$, we have: the partitions $\lambda$ for which $X_{\lambda}$ has a non-zero image in the quotient $\underline{Rep}(S_{\nu \InnaA{=n}})/\idealI_{\nu\InnaA{=n}} \cong Rep(S_n)$ are exactly the $\lambda$ for which $\lambda_1+\abs{\lambda} \leq n$.

If $\lambda_1+\abs{\lambda} \leq n$, then the image of $\lambda$ in $Rep(S_n)$ is the irreducible representation of $S_n$ corresponding to the \InnaG{Young diagram $\tilde{\lambda}(n)$: the Young diagram obtained by adding a row of length $n - \abs{\lambda}$ on top of $\lambda$.}
\mbox{}

\InnaG{ For each $\nu$, we define an equivalence relation $\stackrel{\nu}{\sim}$ on the set of all Young diagrams: we say that $\lambda \stackrel{\nu}{\sim} \lambda'$ if the sequence $ (\nu -\abs{\lambda}, \lambda_1-1, \lambda_2-2, ...)$ can be obtained from the sequence $ (\nu -\abs{\lambda'}, \lambda'_1-1, \lambda'_2-2, ...)$ by permuting a finite number of entries.

The equivalence classes thus obtained are in one-to-one correspondence with the blocks of the category $\underline{Rep}(S_{\nu})$ (see \cite{CO}). 

We say that a block is {\it trivial} if the corresponding equivalence class is trivial, i.e. has only one element (in that case, the block is a semisimple category). 

The non-trivial equivalence classes (respectively, blocks) are parametrized by all Young diagrams of size $\nu$; in particular, this happens only if $\nu \in \bZ_+$. These classes are always of the form $ \{\lambda^{(i)}\}_i$, with
$$\lambda^{(0)} \subset \lambda^{(1)} \subset \lambda^{(2)} \subset ...$$ (each $\lambda^{(i)}$ can be explicitly described based on the Young diagram of size $\nu$ corresponding to this class).}

\subsection{Abelian envelope}\label{ssec:S_nu_abelian_env}

As it was mentioned before, the category $\underline{Rep}(S_{\nu})$ is defined as a Karoubian category. For $\nu \notin \bZ_+$, it is semisimple and thus abelian, but for $\nu \in \bZ_+$, it is not abelian. Fortunately, it has been shown that $\underline{Rep}(S_{\nu})$ possesses an ``abelian envelope'', that is, that it can be embedded (\InnaD{as a full monoidal subcategory}) into an abelian rigid symmetric monoidal category, and this abelian envelope has a universal mapping property (see \cite[Theorem 1.2]{CO2}, \cite[8.21.2]{D}). 

 We will denote the abelian envelope of the Deligne category $\underline{Rep}(S_{\nu})$ by $\underline{Rep}^{ab}(S_{\nu})$ (with $\underline{Rep}^{ab}(S_{\nu}) := \underline{Rep}(S_{\nu})$ for $\nu \notin \bZ_+$).

An explicit construction of the category $\underline{Rep}^{ab}(S_{\nu=n})$ is given in \cite{CO2}, and a detailed description of its structure can be found in \cite{EA}. 
\mbox{}

It turns out that the category $\underline{Rep}^{ab}(S_{\nu})$ is a highest weight category \InnaA{(with infinitely many weights)} corresponding to the partially ordered set $(\{ \text{Young diagrams} \}, \geq )$, where
 $$\lambda \geq \mu \text{ iff } \lambda \stackrel{\nu}{\sim} \mu, \lambda \subset \mu$$
 (namely, in a non-trivial $\stackrel{\nu}{\sim}$-class, $\lambda^{(i)} \geq \lambda^{(j)}$ if $i \leq j$).

Thus the isomorphism classes of simple objects in $\underline{Rep}^{ab}(S_{\nu})$ are parametrized by the set of Young diagrams of arbitrary sizes. We will denote the simple object corresponding to $\lambda$ by $\mathbf{L}(\lambda)$.

We will also use the fact that blocks of the category $\underline{Rep}^{ab}(S_{\nu})$, just like the blocks of $\underline{Rep}(S_{\nu})$, are parametrized by $\stackrel{\nu}{\sim}$-equivalence classes. 
 
 For each $\stackrel{\nu}{\sim}$-equivalence class, 
 \InnaD{the corresponding block of $\underline{Rep}(S_{\nu})$ is \InnaE{the full subcategory of tilting objects in} the corresponding block of $\underline{Rep}^{ab}(S_{\nu})$ (see \cite[Proposition 2.9, Section 4]{CO2}).}

\section{\texorpdfstring{$\gl_{\infty}$}{Infinite Lie algebra gl} and the restricted inverse limit of representations of \texorpdfstring{$\gl_n$}{finite-dimensional Lie algebras gl}}\label{sec:poly_rep}
In this section, we discuss the category of polynomial representations of the Lie algebra $\gl_{\infty}$ and its relation to the categories of polynomial representations of $\gl_n$ for $n \geq 0$. The representations of the Lie algebra $\gl_{\infty}$ are discussed in detail in \cite{PS}, \cite{DPS}, as well as \cite[Section 3]{SS}.

Most of the constructions and the proofs of the statements appearing in this section can be found in \cite[Section 7]{EA1}. 

\subsection{The \texorpdfstring{Lie algebra $\gl_{\infty}$}{infinite Lie algebra gl}}\label{ssec:rep_gl_infty}
Let $\bC^{\infty}$ be a complex vector space with a countable basis $e_1, e_2, e_3, ...$.

Consider the Lie algebra $\gl_{\infty}$ of infinite matrices $A=(a_{ij})_{i, j \geq 1}$ with finitely many non-zero entries. We have a natural action of $\gl_{\infty}$ on $\bC^{\infty}$ \InnaD{and on the restricted dual} $\bC^{\infty}_{*} = span_{\bC}(e_1^*, e_2^*, e_3^*, ...)$ (here $e_i^*$ is the linear functional dual to $e_i$: $e_i^*(e_j) = \delta_{ij}$). 


Let $N \in \bZ_+ \cup \{ \infty \}$, and let $m \geq 1$. We will consider the Lie subalgebra $\gl_m \subset \gl_{N}$ which consists of matrices $A=(a_{ij})_{1 \leq i, j \leq N}$ for which $a_{ij} =0$ whenever $i>m$ or $j>m$. We will also denote by $\gl_m^{\perp}$ the Lie subalgebra of $\gl_{N}$ consisting of matrices $A=(a_{ij})_{1 \leq i, j \leq N}$ for which $a_{ij} =0$ whenever $i\leq m$ or $j \leq m$. 
\begin{remark}
 Note that $\gl_{\InnaD{m}}^{\perp} \cong \gl_{N-m}$ for any $N, m$.
\end{remark}

\subsection{Categories of polynomial representations of \texorpdfstring{$\gl_N$}{gl}}
In this subsection, $N \in \bZ_+ \cup \{\infty \}$. The notation $\bC^N_*$ will stand for $(\bC^N)^*$ whenever $N \in \bZ_+$, and for $\bC^{\infty}_*$ when $N = \infty$.

Consider the category $Rep(\gl_{N})_{poly}$ of polynomial representations of $\gl_{N}$: this is the category of the representations of $\gl_{N}$ which can be obtained as \InnaD{summands} of a direct sum of tensor powers of the tautological representation $\bC^N$ of $\gl_{N}$. 


It is easy to see that this is a semisimple abelian category, whose simple objects are parametrized (up to isomorphism) by all Young diagrams of arbitrary sizes \InnaD{whose length does not exceed $N$}: the simple object corresponding to $\lambda$ is $ S^{\lambda} \bC^{N}$.

\begin{remark}
Note that $Rep(\gl_{\infty})_{poly}$ is the free abelian symmetric monoidal category generated by one object (see \cite[(2.2.11)]{SS}). It has a equivalent definition as the category of polynomial functors of bounded degree, which can be found in \cite{HY} and in \cite{SS}.
\end{remark}

%

Next, we define a natural $\bZ_+$-grading on objects in $Ind-Rep(\gl_{N})_{poly}$ (c.f. \cite[(2.2.2)]{SS}):
\begin{definition}\label{def:grading_poly_rep}
 The objects in $Ind-Rep(\gl_{N})_{poly}$ have a natural $\bZ_+$-grading. Namely, given $M \in Ind-Rep(\gl_{N})_{poly}$, \InnaD{we consider the decomposition} $ M = \bigoplus_{\lambda} S^{\lambda} \bC^N \otimes mult_{\lambda}$ (here $ mult_{\lambda}$ is the multiplicity space of $S^{\lambda} \bC^N$ in $M$), we define $$gr_k(M):= \bigoplus_{\lambda : \abs{\lambda}=k} S^{\lambda} \bC^N \otimes mult_{\lambda}$$

\end{definition}
Of course, the morphisms in $Ind-Rep(\gl_{N})_{poly}$ respect this grading.

%
%
%

\subsection{Specialization and restriction functors}\label{ssec:spec_res_funct_alg}


We now define specialization functors from the category of representations of $\gl_{\infty}$ to the categories of representations of $\gl_n$ (c.f. \cite[Section 3]{SS}):
\begin{definition}\label{def:Gamma_func}
$$\Gamma_n: Rep(\gl_{\infty})_{poly} \rightarrow Rep(\gl_{n})_{poly}, \; \Gamma_n := (\cdot)^{\gl_{n}^{\perp}}$$
\end{definition}

One can easily check (see \cite[Section 7]{EA1}) that the functor $\Gamma_n$ is well-defined. 
%
%
%
%
%
The following Lemma is proved in \cite{PS}, \cite[Section 3]{SS}:
 \begin{lemma}\label{lem:Gamma_is_tensor}
 The functors $\Gamma_n$ are additive symmetric monoidal functors between semisimple symmetric monoidal categories. Their effect on the simple objects is described as follows: for any Young diagram $\lam$, $\Gamma_n(S^{\lam} \bC^{\infty}) \cong S^{\lam} \bC^n$.
\end{lemma}

\InnaD{Next, we define the restriction functors we \InnaE{will} use:}
\begin{definition}\label{def:res_funct_poly_repr}
Let $n \geq 1$. We define the functor $$\InnaD{\mathfrak{Res}}_{n-1, n}: Rep(\gl_n)_{poly} \rightarrow Rep(\gl_{n-1})_{poly}, \; \InnaD{\mathfrak{Res}}_{n-1, n} := (\cdot)^{\gl_{n-1}^{\perp}}$$
\end{definition} 
%
Again, one can easily show that \InnaD{these} functors are well-defined.

\begin{remark}
 There is an alternative definition of the functors $ \InnaD{\mathfrak{Res}}_{n-1, n}$. \InnaD{One can think of the functor $\InnaD{\mathfrak{Res}}_{n-1, n}$ acting on a $\gl_n$-module $M$ as taking the restriction of $M$ to $\gl_{n-1}$ and then considering only the vectors corresponding to ``appropriate'' central characters. 
 
 More specifically, we} say that a $\gl_n$-module $M$ is of {\it degree} $d$ if $\id_{\bC^n} \in \gl_n$ acts by $d \id_M$ on $M$. Also, given any $\gl_n$-module $M$, we may consider the maximal submodule of $M$ of degree $d$, and denote it by $deg_d(M)$. This defines an endo-functor $deg_{d}$ of $Rep(\gl_n)_{poly}$.
 
 Note that a simple module $S^{\lam} \bC^n$ is of degree $\abs{\lambda}$.
 
 \mbox{}
 
 The notion of degree gives a decomposition
 $$Rep(\gl_n)_{poly} \cong \bigoplus_{d \in \bZ_+} Rep(\gl_{n})_{poly, d}$$ where $ Rep(\gl_{n})_{poly, d}$ is the full subcategory of $Rep(\gl_n)_{poly}$ consisting of all polynomial $\gl_n$-modules of degree $d$.

Then
$$ \InnaD{\mathfrak{Res}}_{n-1, n} = \oplus_{d \in \bZ_+} \InnaD{\mathfrak{Res}}_{d, n-1, n}: Rep(\gl_n)_{poly} \rightarrow Rep(\gl_{n-1})_{poly} $$ where
$$\InnaD{\mathfrak{Res}}_{d, n-1, n}: Rep(\gl_{n})_{poly, d} \rightarrow Rep(\gl_{n-1})_{poly, d} , \, \InnaD{\mathfrak{Res}}_{d, n-1, n}:= deg_{d} \circ \InnaD{\mathrm{Res}}_{\gl_{n-1}}^{\gl_n}$$
where $\InnaD{\mathrm{Res}}_{\gl_{n-1}}^{\gl_n}$ is the usual restriction functor for the pair $\gl_{n-1} \subset \gl_n$.
\end{remark}

Once again, the functors $\InnaD{\mathfrak{Res}}_{n-1, n}$ are additive functors between semisimple categories, and satisfy:
\begin{lemma}\label{lem:res_func_simples}
 $\InnaD{\mathfrak{Res}}_{n-1, n}(S^{\lambda} \bC^n) \cong S^{\lambda} \bC^{n-1}$ for any Young diagram $\lam$.
\end{lemma}

Moreover, these functors are compatible with the functors $\Gamma_n$ defined before. 
\begin{lemma}\label{lem:Gamma_res_compat}
 For any $n \geq 1$, we have a commutative diagram:
 $$\xymatrix{&Rep(\gl_{\infty})_{poly}  \ar[r]^{\Gamma_n} \ar[rd]_{\Gamma_{n-1}} &Rep(\gl_n)_{poly} \ar[d]^{\InnaD{\mathfrak{Res}}_{n-1, n}} \\ &{} &Rep(\gl_{n-1})_{poly} }$$
 That is, there is a natural isomorphism $\Gamma_{n-1}  \cong \InnaD{\mathfrak{Res}}_{n-1, n} \circ \Gamma_n $.
\end{lemma}

In particular, this implies: 
\begin{corollary}\label{cor:res_func_tensor}
 The functors $\InnaD{\mathfrak{Res}}_{n-1, n}:Rep(\gl_{n})_{poly} \rightarrow Rep(\gl_{n-1})_{poly}$ are symmetric monoidal functors.
\end{corollary}

\subsection{The restricted inverse limit of categories \texorpdfstring{$Rep(\gl_n)_{poly}$}{of polynomial representations}}\label{ssec:Stab_inv_lim_rep_poly}
This subsection gives a description of the category $Rep(\gl_{\infty})_{poly}$ as a ``restricted'' inverse limit of categories $Rep(\gl_n)_{poly}$ (\InnaD{see \cite{EA1} for details}). 

We will use the \InnaD{framework developed in \cite{EA1}} for the inverse limits of categories \InnaE{with} \InnaD{$\bZ_+$-filtrations on objects}, and the restricted inverse limits of finite-length categories \InnaG{(abelian categories in which every object admits a Jordan-Holder filtration)}. The notions of $\bZ_+$-filtered functors and shortening functors are defined {\it loc. cit.}

We define a $\bZ_+$-filtration on \InnaD{the objects of} $Rep(\gl_n)_{poly}$ for each $n \in \bZ_+$:
\begin{notation}
 For each $k \in \bZ_+$, let $\InnaD{Fil_k(Rep(\gl_n)_{poly})}$ be the full additive subcategory of $Rep(\gl_n)_{poly}$ generated by $S^{\lambda} \bC^n$ such that $\ell(\lambda) \leq k$. 
\end{notation}
Clearly the subcategories $\InnaD{Fil_k(Rep(\gl_n)_{poly})}$ give us a $\bZ_+$-filtration \InnaD{on the objects of} the category $Rep(\gl_n)_{poly}$. \InnaD{Furthermore,} by Lemma \ref{lem:res_func_simples}, the functors $ \InnaD{\mathfrak{Res}}_{n-1, n}$ are $\bZ_+$-filtered functors, \InnaD{i.e. they induce functors $$ \InnaD{\mathfrak{Res}}_{n-1, n}^k: Fil_k(Rep(\gl_n)_{poly}) \longrightarrow Fil_k(Rep(\gl_{n-1})_{poly})$$}

This allows us to consider the inverse limit $$\varprojlim_{n \in \bZ_+, \bZ_+-filtr} Rep(\gl_n)_{poly} \InnaD{\cong \varinjlim_{ k \in \bZ_+} \varprojlim_{n \in \bZ_+} Fil_k(Rep(\gl_n)_{poly})}$$ This is an abelian category (with a natural $\bZ_+$-filtration on objects).

Note that by Lemma \ref{lem:res_func_simples}, the functors $\InnaD{\mathfrak{Res}}_{n-1, n}$ are shortening functors; futhermore, the system $((Rep(\gl_n)_{poly})_{n \in \bZ_+}, (\InnaD{\mathfrak{Res}}_{n-1, n})_{n \geq 1})$ satisfies the conditions \InnaD{listed in} \cite[Section 6]{EA1}, and therefore the \InnaD{category $\varprojlim_{n \in \bZ_+, \bZ_+-filtr} Rep(\gl_n)_{poly}$ is also \InnaE{equivalent to} the} restricted inverse limit \InnaD{of this system, $\varprojlim_{n \in \bZ_+, \text{ restr}}  Rep(\gl_{n})_{poly} $}.

\begin{remark}
 The functors $\InnaD{\mathfrak{Res}}_{n-1, n}$ are symmetric monoidal functors, so the category $\varprojlim_{n \in \bZ_+, \text{ restr}}  Rep(\gl_{n})_{poly} $ is a symmetric monoidal category.
\end{remark}

\begin{proposition}\label{prop:inv_lim_cat_poly_rep}
 We have an equivalence of symmetric monoidal Karoubian categories 
 $$\Gamma_{\text{lim}}: Rep(\gl_{\infty})_{poly} \longrightarrow \varprojlim_{n \in \bZ_+, \text{ restr}}  Rep(\gl_{n})_{poly} $$
 induced by the symmetric monoidal functors 
 $$\Gamma_n =  ( \cdot )^{\gl_n^{\perp}}:  Rep(\gl_{\infty})_{poly} \longrightarrow  Rep(\gl_{n})_{poly}$$
\end{proposition}

\section{Parabolic category \texorpdfstring{$\co$}{O}}\label{sec:par_cat_o}

In this section, we describe a version of the parabolic category $\co$ for $\gl_N$ which we are going to work with. We will give a definition which will describe both the relevant category for $\gl_n$, and for $\gl_{\infty}$. 
\subsection{}
For the benefit of the reader, we will start \InnaD{by} giving a definition for $\gl_N$ when $N$ \InnaD{is a positive integer}; this definition is analogous to the usual definition of the category $\co$. The generic definition will then be just a \InnaD{slight modification} of the first \InnaD{to accomodate the case $N = \infty$}.



\InnaD{This version of the} parabolic category $\co$ is attached to a pair: a vector space $V$ and a fixed non-zero vector $\triv$ in it. Such a pair is called a {\it unital vector space}. In our case, we will just consider $V = \bC^N$, with the standard basis $e_1, e_2, ...$, and the chosen vector $\triv:= e_1$.

\InnaE{Fix $N \in \bZ$, $N \geq 1$.}

The following notation will be used throughout the paper:
\begin{notation}\label{notn:par_subalg}
\mbox{}
\begin{itemize}
 \item \InnaB{We denote by $\mathfrak{p}_{N}\subset \gl_N$ the parabolic Lie subalgebra which consists of} all the endomorphisms $\phi: \bC^N \rightarrow \bC^N$ for which $\InnaB{\phi(\triv) \in \InnaB{\bC \triv} }$. In terms of matrices this is $span\{E_{1,1}, E_{i, j} \rvert j>1 \}$. 
 \item $\mathfrak{u}_{\mathfrak{p}_N}^{+} \subset \mathfrak{p}_N$ denotes the algebra of endomorphisms $\phi: {\bC^N} \rightarrow {\bC^N}$ for which $\Im \phi \subset \InnaB{\bC \triv} \subset \Ker \phi$. In terms of matrices, $\mathfrak{u}_{\mathfrak{p}_N}^{+} = span \{E_{1, j} \rvert j>1 \}$.
\end{itemize}
\end{notation}

Denote $U_N := span \{e_2, e_3, ..., e_N \}$. 

We have a splitting $\gl_N \cong \mathfrak{p}_N \oplus \mathfrak{u}_{\mathfrak{p}_N}^{-}$, where $\mathfrak{u}_{\mathfrak{p}_N}^{-} \cong U_N = span \{E_{i, 1} \rvert i >1 \}$). This gives us an analogue of the triangular decomposition:
$$\gl_N \cong \bC   \id_{\bC^N} \oplus \mathfrak{u}_{\mathfrak{p}_N}^{-} \oplus \mathfrak{u}_{\mathfrak{p}_N}^{+}  \oplus \gl(U_N)$$

We can now give a precise definition of the parabolic category $\co$ which we will use:
\begin{definition}\label{def:par_cat_O_splitting}
 We define the category $\co^{\mathfrak{p}_N}_{\bC^N}$ to be the full subcategory of $Mod_{\mathcal{U}(\gl_N)}$ whose objects $M$ satisfy the following conditions:
 \begin{itemize}
 \item Viewed as a $\mathcal{U}(\gl(U_N))$-module, $M$ is a direct sum of polynomial $\mathcal{U}(\gl(U_N))$-modules (that is, $M$ belongs to $Ind-Rep(\gl(U_N))_{poly}$).
  \item $M$ is locally finite over $\mathfrak{u}_{\mathfrak{p}_N}^{+}$.
  \item \InnaA{$M$ is a finitely generated $\mathcal{U}(\gl_N)$-module.}
 \end{itemize}
\end{definition}
\begin{remark}
 One can replace the requirement that $\mathfrak{u}_{\mathfrak{p}_N}^{+}$ act locally finitely on $M$ by the requirement that $\mathcal{U}(\mathfrak{u}_{\mathfrak{p}_N}^{+})$ act locally nilpotently on $M$.
\end{remark}

\begin{remark}
 One can, in fact, give an equivalent definition of the category $\co^{\mathfrak{p}}_{V}$ corresponding to a finite-dimensional unital vector $(V, \triv)$ without choosing a splitting(c.f. \cite[Section 5]{EA} and Introduction \ref{sec:intro}).
\end{remark}

\InnaD{We now recall the following definition:}
\begin{definition}\label{def:degree_gl_module}
 A module $M$ over the Lie algebra $\gl_N$ will be said to be of \InnaA{degree} $K \in \bC$ if $\id_{\bC^N} \in \gl_N$ acts by $K   \id_M$ on $M$.
\end{definition}

We will denote by $\co^{\mathfrak{p}_N}_{\nu, \bC^N}$ the full subcategory of $\co^{\mathfrak{p}_N}_{\bC^N}$ whose objects are modules of \InnaA{degree} $\nu$. Note that for a module $M$ of $\co^{\mathfrak{p}_N}_{\bC^N}$ to be of degree $\nu$ is the same as to require that $E_{1,1}$ acts on each subspace $S^{\lambda} U_N$ of $M$ by the scalar $\nu - \abs{\lambda}$.

\begin{definition}\label{def:deg_functor}
Let $\nu \in \bC$. Define the functor $deg_{\nu}:Mod_{\mathcal{U}(\gl_N)} \rightarrow Mod_{\mathcal{U}(\gl_N)}$ by putting $deg_{\nu}(E)$ to be the maximal submodule of $E$ of \InnaA{degree} $\nu$ (see Definition \ref{def:degree_gl_module}). For a morphism $f: E \rightarrow E'$ of $\gl_N$-modules, we put $deg_{\nu}(f):= f\rvert_{deg_{\nu}(E)}$.
\end{definition}

Let $E \in Mod_{\mathcal{U}(\gl_N)}$. The maximal submodule of $E$ of \InnaA{degree} $\nu$ is well-defined: it is the subspace of $E$ consisting of all vectors on which $\id_{\bC^N}$ acts by the scalar $\nu$, and it is a $\gl_N$-submodule since $\id_{\bC^N}$ lies in the center of $\gl_N$.

One can show that the functor $deg_{\nu}:Mod_{\mathcal{U}(\gl_N)} \rightarrow Mod_{\mathcal{U}(\gl_N)}$ is left-exact. Moreover, it is easy to show that the category $\co^{\mathfrak{p}_N}_{\nu, \bC^N}$ is a direct summand of $\co^{\mathfrak{p}_N}_{\bC^N}$, and the functor $deg_{\nu}:\co^{\mathfrak{p}_N}_{\bC^N} \rightarrow \co^{\mathfrak{p}_N}_{\nu, \bC^N}$ is exact.

\subsection{Parabolic category \texorpdfstring{$\co$ for $\gl_{N}$}{O for the Lie algebra gl}}\label{ssec:def_par_cat_O_gl_N}
We now give a definition of the parabolic category $\co$ which for $\gl_N$, $N$ being either a positive integer or infinity.

\mbox{}

Again, we let $N \in \bZ_{\geq 1} \cup \{\infty \}$. 

Consider a unital vector space $(\bC^N, \triv)$, where $\triv:=e_1$. Put $U_{N}:=span_{\bC}(e_2, e_3, ...) \subset \bC^N$, so that we have a splitting $\bC^N = \bC e_1 \oplus U_N$. We will also denote $U_{N, *} := span (e_2^*, e_3^*, ...)$ (so $U_{N, *} = U_N^*$ whenever $N \in \bZ$).

We have a decomposition $$\gl_{N}  \cong \gl(U_N) \oplus \gl_1  \oplus \mathfrak{u}_{\mathfrak{p}_N}^{+}  \oplus \mathfrak{u}_{\mathfrak{p}_N}^{-}$$
Of course, for any $N$, $\mathfrak{u}_{\mathfrak{p}_N}^{-} \cong U_N$; moreover, $\mathfrak{u}_{\mathfrak{p}_N}^{+}  \cong U_{N, *}$.

We will also use the isomorphisms $\gl(U_N ) \cong \gl^{\perp}_1 \cong \gl_{N-1}$.

\begin{definition}\label{def:par_cat_O_general_1}
\mbox{}
\begin{itemize}[leftmargin=*]
%
\item Define the category $Mod_{\gl_{N}, \gl(U_N) -poly}$ to be the category of $\gl_{N}$-modules whose restriction to $\gl(U_N)$ lies in $Ind-Rep(\gl_{U_N})_{poly}$; that is, $\gl_{N}$-modules whose restriction to $\gl(U_N)$ is a (perhaps infinite) direct sum of Schur functors applied to $U_N$. 

The morphisms would be $\gl_{N}$-equivariant maps. 

\mbox{}

\item
 We say that an object $M \in Mod_{\gl_{N}, \gl(U_N) -poly}$ is of {\it degree} $\nu$ ($\nu \in \bC$) if on every summand $S^{\lambda} U_N \InnaD{\subset M}$, the element $E_{1,1} \in \gl_{N}$ acts by $(\nu - \abs{\lambda}) \id_{S^{\lambda} U_N}$. 
 
 \mbox{}
 
 \item Let $M \in Mod_{\gl_{N}, \gl(U_N) -poly}$. We have a commutative algebra $Sym(U_N) \cong \mathcal{U}(\mathfrak{u}_{\mathfrak{p}_N}^{-})$ (the enveloping algebra of $\mathfrak{u}_{\mathfrak{p}_N}^{-} \subset \gl_{N}$). The action of $\gl_N$ on $M$ gives $M$ a structure of a $Sym(U_N)$-module. 
 
 We say that $M$ is {\it finitely generated} over $Sym(U_N)$ if $M$ is a quotient of a ``free finitely-generated $Sym(U_N)$-module''; that is, as a $Sym(U_N)$-module, $M$ is a quotient (in $Ind-Rep(\gl_{N})_{poly}$) of $Sym(U_N) \otimes E$ for some $E \in Rep(\gl \InnaD{(U_{N})})_{poly}$.
 
 \mbox{}
 
 \item Let $M \in Mod_{\gl_{N}, \gl(U_N) -poly}$. We have a commutative algebra $Sym(U_{N, *}) \cong \mathcal{U}(\mathfrak{u}_{\mathfrak{p}_N}^{+})$ (the enveloping algebra of $\mathfrak{u}_{\mathfrak{p}_N}^{+} \subset \gl_{N}$). The action of $\gl_N$ on $M$ gives $M$ a structure of a $Sym(U_{N, *})$-module.  
 
 We say that $M$ is {\it locally nilpotent} over the \InnaD{algebra $\mathcal{U}(\mathfrak{u}_{\mathfrak{p}_N}^{+})$} if for any $v \in M$, there exists $m \geq 0$ such that \InnaD{for any $A \in Sym^m (U_{N, *})$ we have:} $A.v =0$.

\end{itemize}
 
\end{definition}

Recall the natural $\bZ_+$-grading on the object of $Ind-Rep(\gl_{N})_{poly}$. 

For each $M \in Mod_{\gl_{N}, \gl(U_N) -poly}$, the above definition implies: $\gl(U_N)$ acts by operators act by operators of degree zero, $U_{N, *}$ acts by operators of degree $1$.
%
We now define the parabolic category $\co$ for $\gl_N$ which we will use throughout the paper: 
\begin{definition}\label{def:par_cat_O_gl_N}
 We define the category $\co^{\mathfrak{p}_{N}}_{\nu, \bC^{N}}$ to be the full subcategory of $Mod_{\gl_{N}, \gl(U_N) -poly}$ whose objects $M$ satisfy the following requirements:
 \begin{itemize}
  \item $M$ is of degree $\nu$.
  \item $M$ is finitely generated over $Sym(U_N)$.
 \item $M$ is locally nilpotent over the \InnaD{algebra $\mathcal{U}(\mathfrak{u}_{\mathfrak{p}_N}^{+})$}.
 \end{itemize}

\end{definition}

Of course, for \InnaE{a positive integer $N$}, this is just the category $\co^{\mathfrak{p}_{N}}_{\nu, \bC^{N}}$ we defined in the beginning of this section.

\mbox{}

We will also consider the localization of the category $\co^{\mathfrak{p}_{N}}_{\nu, \bC^{N}}$ by its Serre subcategory of polynomial $\gl_N$-modules of \InnaD{degree} $\nu$; such modules exist iff $\nu \in \bZ_+$. This localization will be denoted by 
 $$\hat{\pi}_N: {\co}^{\mathfrak{p}_{N}}_{\nu, \bC^{N}} \longrightarrow \widehat{\co}^{\mathfrak{p}_{N}}_{\nu, \bC^{N}} $$
and will play \InnaD{an important} role when we consider the Schur-Weyl duality in complex rank.

\subsection{Duality in category \texorpdfstring{$\co$}{O}}\label{ssec:duality_cat_O}

Let $n \in \bZ_+$.

Recall that in the category $\co$ for $\gl_n$ we have the notion of a duality (c.f. \cite[Section 3.2]{H}): namely, given a $\gl_n$-module $M$ with finite-dimensional weight spaces, we can consider the twisted action of $\gl_n$ on the dual space $M^*$, given by $A.f := f \circ A^{T}$, where $A^T$ means the transpose of $A \in \gl_n$. This makes $M^*$ a $\gl_n$-module. We then take $M^{\vee}$ to be the maximal submodule of $M^{*}$ lying in category $\co$. 

More explicitly, considering $M$ as a direct sum of its finite-dimensional weight spaces
$$M = \bigoplus_{\lambda} M_{\lambda}$$ we can consider the restricted twisted dual
$$M^{\vee} := \bigoplus_{\lambda} M_{\lambda}^*$$
(that is, we take the dual to each weight space separately). The action of $\gl_n$ is given by $A.f := f \circ A^{T}$ for any $A \in \gl_n$.

The module $M^{\vee}$ is called the dual of $M$, and we get an exact functor $(\cdot)^{\vee}: \co^{op} \rightarrow \co$. 

\begin{proposition}
 The category $\co^{\mathfrak{p}_n}_{\bC^n}$ is closed under taking duals, and the duality functor $(\cdot)^{\vee}: \left(\co^{\mathfrak{p}_n}_{\bC^n}\right)^{op} \rightarrow  \co^{\mathfrak{p}_n}_{\bC^n} $ is an equivalence of categories.
\end{proposition}

In fact, a similar construction can be made for $\co^{\mathfrak{p}_{\infty}}_{\nu, \bC^{\infty}}$. All modules $M$ in $\co^{\mathfrak{p}_{\infty}}_{\nu, \bC^{\infty}}$ are weight modules with respect to the subalgebra of diagonal matrices in $\gl_{\infty}$, and the weight spaces are finite-dimensional (due to the polynomiality condition in the definition of $\co^{\mathfrak{p}_{\infty}}_{\nu, \bC^{\infty}}$). This allows one to construct the restricted twisted dual $M^{\vee}$ in the same way as before, and obtain an exact functor $$(\cdot)^{\vee}: \left(\co^{\mathfrak{p}_{\infty}}_{\nu, \bC^{\infty}}\right)^{op} \longrightarrow \co^{\mathfrak{p}_{\infty}}_{\nu, \bC^{\infty}}$$

\begin{remark}
 It is obvious that for $n \in \bZ_+$, the functor $(\cdot)^{\vee}: \left(\co^{\mathfrak{p}_n}_{\bC^n}\right)^{op} \rightarrow  \co^{\mathfrak{p}_n}_{\bC^n} $ takes finite-dimensional \InnaD{(polynomial)} modules to finite-dimensional \InnaD{(polynomial)} modules.
 
 In fact, one can easily check that the functor $(\cdot)^{\vee}: \left(\co^{\mathfrak{p}_{\infty}}_{\nu, \bC^{\infty}}\right)^{op} \longrightarrow \co^{\mathfrak{p}_{\infty}}_{\nu, \bC^{\infty}}$ takes polynomial modules to polynomial modules as well. 
\end{remark}

\subsection{Structure of the category \texorpdfstring{$\co^{\mathfrak{p}_n}_{\nu, \bC^n}$}{parabolic category O}}\label{ssec:structure_cat_O}

In this subsection, we present some facts about the category \InnaD{$\co^{\mathfrak{p}_n}_{\nu, \bC^n} $} which will be used later on. The material of this section is discussed in more detail in \cite[Section 5]{EA} \InnaD{and is} mostly based on \cite[Chapter 9]{H}.

Fix $\nu \in \bC$, and fix $n \in \bZ_+$. We denote by $e_1, e_2, ..., e_n$ the standard basis of $\bC^n$, and denote $\triv:= e_1$, $U_{\InnaD{n}} := span \{ e_2, e_3, ..., e_n\}$.

We will consider the category $\co^{\mathfrak{p}_n}_{{\bC^n}}$ for the unital vector space $(\bC^n, \triv)$ and the splitting $\bC^n = \bC \triv \oplus U_n$.


\begin{proposition}
 The category $\co^{\mathfrak{p}_n}_{{\bC^n}}$ (resp. $Ind-\co^{\mathfrak{p}_n}_{{\bC^n}}$) is closed under taking duals, direct sums, submodules, quotients and extensions in $\co_{\gl_n}$, as well as tensoring with finite dimensional $\gl_n$-modules. 
\end{proposition}

The category $\co^{\mathfrak{p}_n}_{\nu, \bC^n}$ decomposes into blocks (each of the blocks is an abelian category in its own right). To each $\stackrel{\nu}{\sim}$-class of Young diagrams corresponds a block of $\co^{\mathfrak{p}}_{\nu, \bC^n}$. If all Young diagrams $\lambda$ in this $\stackrel{\nu}{\sim}$-class have length at least $n$, then the corresponding block is zero. To each non-zero block of $\co^{\mathfrak{p}}_{\nu, \bC^n}$ corresponds a unique $\stackrel{\nu}{\sim}$-class. 

Moreover, the blocks corresponding to trivial $\stackrel{\nu}{\sim}$-classes are either semisimple (i.e. equivalent to the category $Vect_{\bC}$), or zero. 

\mbox{}

We now proceed to discuss standard objects in $\co^{\mathfrak{p}_n}_{\bC^n}$.
\begin{definition}\label{def:parabolic_Verma_mod}
Let $\lambda$ be a Young diagram. \InnaD{The generalized Verma module} $M_{\mathfrak{p}_n}(\nu-\abs{\lambda}, \lambda)$ is defined to be the $\gl_n$-module $$\mathcal{U}(\gl_n) \otimes_{\mathcal{U}(\mathfrak{p}_{\InnaD{n}})} S^{\lambda} U_{\InnaD{n}}$$ where $\gl(U_{\InnaD{n}})$ acts naturally on $S^{\lambda} U_{\InnaD{n}}$, $\id_{\bC^n} \in \mathfrak{p}_{\InnaD{n}}$ acts on $S^{\lambda} U_{\InnaD{n}}$ by scalar $\nu$, and $\mathfrak{u}_{\mathfrak{p}_n}^{+}$ acts on $S^{\lambda} U_{\InnaD{n}}$ by zero.
 
 Thus $M_{\mathfrak{p}_n}(\nu-\abs{\lambda}, \lambda)$ is the parabolic Verma module for $(\mathfrak{gl}_n, \mathfrak{p}_{\InnaD{n}})$ with highest weight $(\nu-\abs{\lambda}, \lambda)$ iff $n-1 \geq \ell(\lambda)$, and zero otherwise.
\end{definition}


\begin{definition}
 $L(\nu-\abs{\lambda}, \lambda)$ is defined to be zero (if $n \geq \ell(\lambda)$), or the simple module for $\mathfrak{gl}_n$ of highest weight $(\nu-\abs{\lambda}, \lambda)$ otherwise.
\end{definition}

The following basic lemma will be very helpful:
\begin{lemma}\label{lem:gl_u_struct_o_cat}
 Let $\lambda$ \InnaD{be a Young diagram} such that $  \ell(\lambda) < n$. We then have an isomorphism of $\gl(U_{\InnaD{n}})$-modules:
$$M_{\mathfrak{p}_n}(\nu-\abs{\lambda}, \lambda) \cong Sym(U_{\InnaD{n}}) \otimes S^{\lambda} U_{\InnaD{n}}$$
\end{lemma}


%
%

We will also use the \InnaE{following lemma.}

\begin{lemma}\label{lem:parab_Verma_ses}
Let $ \{\lambda^{(i)}\}_i$ be a non-trivial $\stackrel{\nu}{\sim}$-class, and $i \geq 0$ be such that $ \ell(\lambda^{(i)}) <n $.

Then there is a short exact sequence
$$ 0 \rightarrow L(\nu-\abs{\lambda^{(i+1)}}, \lambda^{(i+1)}) \rightarrow M_{\mathfrak{p}_n}(\nu-\abs{\lambda^{(i)}}, \lambda^{(i)}) \rightarrow L(\nu-\abs{\lambda^{(i)}}, \lambda^{(i)}) \rightarrow 0$$

\end{lemma}

\begin{corollary}\label{cor:simple_Verma_basis_Groth_group}
 \InnaE{\InnaF{The isomorphism classes} of the generalized Verma modules and the simple polynomial modules in $\co^{\mathfrak{p}_n}_{\nu, {\bC^n}}$ form a basis for the Grothendieck group of $\co^{\mathfrak{p}_n}_{\nu, {\bC^n}}$. }
\end{corollary}

\section{Stable inverse limit of parabolic categories \texorpdfstring{$\co$}{O}}\label{sec:lim_par_cat_O}
\subsection{Restriction functors}\label{ssec:Res_functors}
%
%
%
%
%


\begin{definition}\label{def:Res_funct_cat_o}
Let $n \geq 1$. Define the functor $$\InnaD{\mathfrak{Res}}_{n-1, n}: \co^{\mathfrak{p}_{n}}_{\nu, \bC^{n}} \longrightarrow \co^{\mathfrak{p}_{n-1}}_{\nu, \bC^{n-1}}, \;\; \InnaD{\mathfrak{Res}}_{n-1, n} := ( \cdot)^{\gl_{n-1}^{\perp}}$$
\end{definition} 
\InnaD{Again}, the subalgebras $\gl_{n-1}, \gl_{n-1}^{\perp} \subset \gl_n$ commute, and therefore the subspace of $\gl_{n-1}^{\perp}$-invariants of a $\gl_n$-module automatically carries an action of $\gl_{n-1}$.

\mbox{}

We need to check that this functor is well-defined. In order to do \InnaD{so}, consider the functor $\InnaD{\mathfrak{Res}}_{n-1, n}: \co^{\mathfrak{p}_{n}}_{\nu, \bC^{n}} \longrightarrow Mod_{\mathcal{U}(\gl_{n-1})}$. This functor is well-defined, and we will show that the objects in the image lie in the full subcategory $\co^{\mathfrak{p}_{n-1}}_{\nu, \bC^{n-1}}$ of $Mod_{\mathcal{U}(\gl_{n-1})}$.

Note that the functor $\InnaD{\mathfrak{Res}}_{n-1, n}$ can alternatively be defined as follows: for a module $M$ in $\co^{\mathfrak{p}_n}_{\nu, \bC^n}$, we restrict the action of $\gl_n$ to $\gl_{n-1}$, and then only take the vectors in $M$ attached to specific central characters. More specifically, we have:
\begin{lemma}\label{lem:two_equiv_def_Res}
 The functor $\InnaD{\mathfrak{Res}}_{n-1, n}$ is naturally isomorphic to the composition $ deg_{\nu} \circ \InnaA{\InnaD{\mathrm{Res}}}^{\gl_n}_{\gl_{n-1}}$ \InnaA{(the functor $ deg_{\nu} $ was defined in Definition \ref{def:deg_functor})}.
\end{lemma}
\begin{proof}
Let $M \in \co^{\mathfrak{p}_n}_{\nu, \bC^n}$. For any vector $m \in M$, we know that $\id_{\bC^n}.m = (E_{1,1} + E_{2,2} + ... +E_{n,n}).m = \nu m$. Then the requirement that $\id_{\bC^{n-1}}.m = (E_{1,1} + E_{2,2} + ... +E_{n-1,n-1}).m = \nu m$ is equivalent to requiring that $E_{n,n}.m = 0$, namely that $m \in  M^{\gl_{n-1}^{\perp}}$.
\end{proof}

We will now use this information to prove the following lemma:

\begin{lemma}\label{lem:Res_defined}
The functor $\InnaD{\mathfrak{Res}}_{n-1, n}: \co^{\mathfrak{p}_{n}}_{\nu, \bC^{n}} \longrightarrow \co^{\mathfrak{p}_{n-1}}_{\nu, \bC^{n-1}}$ is well-defined.
\end{lemma}
\begin{proof}
Let $M \in \co^{\mathfrak{p}_n}_{\nu, \bC^n}$, and consider the $\gl_{n-1}$-module $\InnaD{\mathfrak{Res}}_{n-1, n}(M)$. \InnaE{By definition, this is a module of degree $\nu$. We will show that it lies in $ \co^{\mathfrak{p}_{n-1}}_{\nu, \bC^{n-1}}$.}
%

\InnaD{First of all, consider \InnaD{the} inclusion $\InnaD{\gl(U_{n-1})^{\perp}} \oplus \gl(U_{n-1}) \subset \gl(U_n)$. This inclusion gives us the restriction functor (see Definition \ref{def:res_funct_poly_repr}) $$\InnaD{\mathfrak{Res}}_{ \InnaD{U_{n-1}, U_n}}: Rep(\gl(U_n))_{poly} \longrightarrow Rep(\gl(U_{n-1}))_{poly} , \;\; \InnaD{\mathfrak{Res}}_{ \InnaD{U_{n-1}, U_n}} := (\cdot)^{\InnaD{\gl(U_{n-1})^{\perp}}}$$ The latter is an additive functor between semisimple categories, and takes polynomial representations of $\gl(U_n)$ to polynomial representations of $\gl(U_{n-1})$.}

Now, the restriction to $\gl(U_{n-1})$ of \InnaD{the $\gl_{n-1}$-module} $\InnaD{\mathfrak{Res}}_{n-1, n}(M)$ is isomorphic to $\InnaD{\mathfrak{Res}}_{\InnaD{U_{n-1}, U_n}}(M \rvert_{\gl(U_n)})$, and thus is a polynomial representation of $\gl(U_{n-1})$. 

Secondly, $\InnaD{\mathfrak{Res}}_{n-1, n}(M)$ is locally nilpotent over $\InnaD{\mathcal{U}(\mathfrak{u}_{\mathfrak{p}_{n-1}}^{+})}$, since $M$ is locally nilpotent over $\InnaD{\mathcal{U}(\mathfrak{u}_{\mathfrak{p}_{n}}^{+})}$ and $\InnaD{\mathcal{U}(\mathfrak{u}_{\mathfrak{p}_{n-1}}^{+}) \subset \mathcal{U}(\mathfrak{u}_{\mathfrak{p}_{n}}^{+})}$.


\mbox{}

\InnaE{It remains} to check that given $M \in \co^{\mathfrak{p}_{n}}_{\nu, \bC^{n}}$, the module $\InnaD{\mathfrak{Res}}_{n-1, n}(M)$ is finitely generated over $Sym(U_{n-1})$. Indeed, we know that there exists a polynomial $\gl(U_n)$-module $E$ and a surjective $\gl(U_n)$-equivariant morphism of $Sym(U_n)$-modules $Sym(U_n) \otimes E \twoheadrightarrow M$. Taking the $\gl(U_{n-1})^{\perp}$-invariants and using Lemma \ref{cor:res_func_tensor}, we conclude that there is a surjective $\gl(U_{n-1})$-equivariant morphism of $Sym(U_{n-1})$-modules 
$$Sym(U_{n-1}) \otimes E^{\gl(U_{n-1})^{\perp}} \twoheadrightarrow \InnaD{\mathfrak{Res}}_{n-1, n}(M)$$ 
Thus $\InnaD{\mathfrak{Res}}_{n-1, n}(M)$ is finitely generated over $Sym(U_{n-1})$.

%
%
%
 
\end{proof}

\begin{lemma}\label{lem:Res_exact}
The functor $\InnaD{\mathfrak{Res}}_{n-1, n}: \co^{\mathfrak{p}_{n}}_{\nu, \bC^{n}} \longrightarrow \co^{\mathfrak{p}_{n-1}}_{\nu, \bC^{n-1}}$ is exact.
\end{lemma}
\begin{proof}
We use Lemma \ref{lem:two_equiv_def_Res}. The functor $deg_{\nu}: \co^{\mathfrak{p}_{n-1}}_{\bC^{n-1}} \rightarrow \co^{\mathfrak{p}_{n-1}}_{\nu, \bC^{n-1}}$ is exact, so the functor $\InnaA{\InnaD{\mathrm{Res}}}_{n-1, n}$ is obviously exact as well.

%
%
\end{proof}

\begin{lemma}\label{lem:Res_image_Vermas}
 The functor $\InnaD{\mathfrak{Res}}_{n-1, n}$ takes parabolic Verma modules to either parabolic Verma modules or to zero:
 $$\InnaD{\mathfrak{Res}}_{n-1, n}(M_{\mathfrak{p}_n}(\nu-\abs{\lambda}, \lambda))\cong M_{\mathfrak{p}_{n-1}}(\nu-\abs{\lambda}, \lambda)$$ (recall that the latter is a parabolic Verma module for $\gl_{n-1}$ iff $\ell(\lambda) \leq n-2$, and zero otherwise).
 
\end{lemma}
\begin{proof}
Consider the parabolic Verma module $M_{\mathfrak{p}_n}(\nu-\abs{\lambda}, \lambda)$, \InnaD{where} the Young diagram $\lambda$ \InnaD{has} length at most $n-1$.

By \InnaA{definition} of the parabolic Verma module $M_{\mathfrak{p}_n}(\nu-\abs{\lambda}, \lambda)$, we have: $$M_{\mathfrak{p}_n}(\nu-\abs{\lambda}, \lambda) = \mathcal{U}(\gl_n) \otimes_{\mathcal{U}(\mathfrak{p}_{\InnaD{n}})} S^{\lambda} U_n$$
The branching rule for $\gl(U_{n-1}) \subset \gl(U_n)$ \InnaD{tells} us that $$\left( S^{\lambda} U_n \right) \rvert_{\gl(U_{n-1})} \cong \bigoplus_{\lambda' } S^{\lambda'} U_{n-1}$$
\InnaE{the sum taken over} \InnaD{the set of all Young diagrams obtained from $\lambda$ by removing several boxes, no two in the same column.}

So $$\InnaA{\InnaD{\mathrm{Res}}}^{\gl_n}_{\gl_{n-1}}(M_{\mathfrak{p}_n}(\nu-\abs{\lambda}, \lambda)) \cong \left( \bigoplus_{\lambda' \subset \lambda } M_{\mathfrak{p}_{n-1}}(\nu-\abs{\lambda}, \lambda') \right) \otimes \mathcal{U} \left( \quotient{\mathfrak{u}_{\mathfrak{p}_{n}}^{-}}{\mathfrak{u}_{\mathfrak{p}_{n-1}}^{-}} \right)$$
Here 
\begin{itemize}
 \item $ M_{\mathfrak{p}_{n-1}}(\nu-\abs{\lambda}, \lambda')$ is either a parabolic Verma module for $\gl_{n-1}$ of highest weight $(\nu-\abs{\lambda}, \lambda')$ (note that it is \InnaD{of degree $\nu- \abs{\lambda} + \abs{\lambda'}$}) or zero.
 \item $\InnaD{\gl(U_{n-1})}$ acts trivially on the space $ \mathcal{U} \left( \quotient{\mathfrak{u}_{\mathfrak{p}_{n}}^{-}}{\mathfrak{u}_{\mathfrak{p}_{n-1}}^{-}} \right)$. This space is isomorphic, as a $\bZ_+$-graded vector space, to $\bC[t]$ ($t$ standing for $E_{n, 1} \in \gl_n$) and \InnaD{$E_{1,1}$ acts on it by derivations $-t \frac{d}{dt}$.}
\end{itemize}


\InnaD{Thus $\id_{\bC^{n-1}} \in \gl_{n}$ acts on the subspace $M_{\mathfrak{p}_{n-1}}(\nu-\abs{\lambda}, \lambda') \otimes t^k \subset M_{\mathfrak{p}_n}(\nu-\abs{\lambda}, \lambda)$} by the scalar $\nu - \abs{\lambda} + \abs{\lambda'} -k$. 

We now apply the functor $deg_{\nu}$ to the module $\InnaE{\mathrm{Res}^{\gl_n}_{\gl_{n-1}}}(M_{\mathfrak{p}_n}(\nu-\abs{\lambda}, \lambda))$.

To see which subspaces $M_{\mathfrak{p}_{n-1}}(\nu-\abs{\lambda'}, \lambda') \otimes t^k$ of $M_{\mathfrak{p}_n}(\nu-\abs{\lambda}, \lambda)$ will survive after \InnaD{applying $deg_{\nu}$}, we require that $\abs{\lambda} - \abs{\lambda'} +k = 0$. But we are only considering Young diagrams $\lambda'$ such that $\lambda' \subset \lambda$, and non-negative integers $k$, which means that the only relevant case is $\lambda'=\lambda$, $k=0$.

We conclude that $$\InnaD{\mathfrak{Res}}_{n-1, n}(M_{\mathfrak{p}_N}(\nu-\abs{\lambda}, \lambda))\cong M_{\mathfrak{p}_{n-1}}(\nu-\abs{\lambda}, \lambda)$$
\end{proof}
%
 
 \begin{lemma}\label{lem:Res_image_simples}
  \InnaD{Given a simple $\gl_n$-module $L_n(\nu-\abs{\lambda}, \lambda)$, $$\InnaD{\mathfrak{Res}}_{n-1, n}(L_n(\nu-\abs{\lambda}, \lambda))\cong L_{n-1}(\nu-\abs{\lambda}, \lambda)$$ (recall that the latter is a simple $\gl_{n-1}$-module iff $\ell(\lambda) \leq n-2$, and zero otherwise).}
 \end{lemma}
\begin{proof}
\InnaD{Note that the statement} follows immediately from Lemma \ref{lem:Res_image_Vermas} when $\lambda$ lies in a trivial $\stackrel{\nu}{\sim}$-class; for a non-trivial $\stackrel{\nu}{\sim}$-class $\{\lambda^{(i)}\}_i$, we have short exact sequences (see Lemma \ref{lem:parab_Verma_ses}):
 $$ 0 \rightarrow L_n(\nu-\abs{\lambda^{(i+1)}}, \lambda^{(i+1)}) \rightarrow M_{\mathfrak{p}_n}(\nu-\abs{\lambda^{(i)}}, \lambda^{(i)}) \rightarrow L_n(\nu-\abs{\lambda^{(i)}}, \lambda^{(i)}) \rightarrow 0$$

 Using the exactness of $\InnaD{\mathfrak{Res}}_{n-1, n}$, we can prove the required statement for 
 
 $L_n(\nu-\abs{\lambda^{(i)}}, \lambda^{(i)})$ by induction on $i$, provided the statement is true for $i=0$.
 
 \mbox{}
 
 So it remains to check that 
 $$\InnaD{\mathfrak{Res}}_{n-1, n}(L_n(\nu-\abs{\lambda}, \lambda))\cong L_{n-1}(\nu-\abs{\lambda}, \lambda)$$
 for the minimal Young diagram $\lambda$ in any non-trivial $\stackrel{\nu}{\sim}$-class.

 Recall that in that case, $L_n(\nu-\abs{\lambda}, \lambda) = S^{\tilde{\lambda}(\nu)} \bC^n$ is a finite-dimensional simple representation of $\gl_n$. 
 
 The branching rule for $\gl_n, \gl_{n-1}$ implies that $$\InnaD{\mathrm{Res}^{\gl_n}_{\gl_{n-1}} ( S^{\tilde{\lambda}(\nu)} \bC^n ) \cong \bigoplus_{\mu } S^{\mu} \bC^{n-1} }$$ \InnaE{the sum taken over} \InnaD{the set of all Young diagrams obtained from $\tilde{\lambda}(\nu)$ by removing several boxes, no two in the same column.}
 
 Considering only the summands of degree $\nu$, we see that 
 $$\InnaD{\mathfrak{Res}}_{n-1, n}(L_n(\nu-\abs{\lambda}, \lambda))\cong S^{\tilde{\lambda}(\nu)} \bC^{n-1} \cong L_{n-1}(\nu-\abs{\lambda}, \lambda)$$
 
\end{proof}

 The functor $\InnaD{\mathfrak{Res}}_{n-1, n}: \co^{\mathfrak{p}_n}_{\nu, \bC^n} \rightarrow \co^{\mathfrak{p}_{n-1}}_{\nu, \bC^{n-1}} $ clearly takes polynomial modules to polynomial modules; together with Lemma \ref{lem:Res_exact}, this means that $\InnaD{\mathfrak{Res}}_{n-1, n}$ factors through an exact functor $$\widehat{\InnaD{\mathfrak{Res}}}_{n-1, n}: \widehat{\co}^{\mathfrak{p}_n}_{\nu, \bC^n} \rightarrow \widehat{\co}^{\mathfrak{p}_{n-1}}_{\nu, \bC^{n-1}} $$
 i.e. we have a commutative diagram
 $$ \begin{CD}
     \co^{\mathfrak{p}_n}_{\nu, \bC^n} @>{\InnaD{\mathfrak{Res}}_{n-1, n}}>> \co^{\mathfrak{p}_{n-1}}_{\nu, \bC^{n-1}}\\
     @V{\hat{\pi}_n}VV @V{\hat{\pi}_{n-1}}VV \\
     \widehat{\co}^{\mathfrak{p}_n}_{\nu, \bC^n} @>{\widehat{\InnaD{\mathfrak{Res}}}_{n-1, n}}>> \widehat{\co}^{\mathfrak{p}_{n-1}}_{\nu, \bC^{n-1}}
    \end{CD}$$
(see Subsection \ref{ssec:def_par_cat_O_gl_N} for the definition of the localizations $\hat{\pi}_n$).

\subsection{Specialization functors}\label{ssec:Spec_functors}

\begin{definition}\label{def:Gamma_funct_cat_o}
Let $n \geq 1$. Define the functor $$\Gamma_{n}: \co^{\mathfrak{p}_{\infty}}_{\nu, \bC^{\infty}} \longrightarrow \co^{\mathfrak{p}_{n}}_{\nu, \bC^{n}}, \;\; \Gamma_{n}:= ( \cdot)^{\gl_{n}^{\perp}}$$
\end{definition} 
As before, the subalgebras $\gl_{n}, \gl_{n}^{\perp} \subset \gl_{\infty}$ commute, and therefore the subspace of $\gl_{n}^{\perp}$-invariants of a $\gl_{\infty}$-module automatically carries an action of $\gl_{n}$.

\mbox{}



\begin{lemma}\label{lem:Gamma_defined}
The functor $\Gamma_{ n}: \co^{\mathfrak{p}_{\infty}}_{\nu, \bC^{\infty}} \longrightarrow \co^{\mathfrak{p}_{n}}_{\nu, \bC^{n}}$ is well-defined.
\end{lemma}
\begin{proof}
The proof is essentially the same as in Lemma \ref{lem:Res_defined}.

\end{proof}
Next, we check that the functor $\Gamma_n$ is exact:
\begin{lemma}\label{lem:Gamma_exact}
The functor $\Gamma_{n}: \co^{\mathfrak{p}_{\infty}}_{\nu, \bC^{\infty}} \longrightarrow \co^{\mathfrak{p}_{n}}_{\nu, \bC^{n}}$ is exact.
\end{lemma}
\begin{proof}
The definition of $\Gamma_{n}$ immediately implies that this functor is left-exact. 
\InnaD{Consider the \InnaE{inclusion $\gl(U_{n}) \oplus \gl(U_{n})^{\perp} \subset \gl(U_{\infty})$}. We then have an isomorphism of $\gl(U_n)$-modules $$ \left( M \rvert_{\gl(U_{\infty})} \right)^{\gl(U_n)^{\perp}} \cong \left( M^{\gl_n^{\perp}} \right) \rvert_{\gl(U_n)} $$

The exactness of $\Gamma_{n}$ then follows from the additivity of the functor}
%
 
$(\cdot)^{\gl(U_{n})^{\perp}}: Rep(\gl(U_{\infty}))_{poly} \rightarrow Rep(\gl(U_n))_{poly}$, \InnaD{which} is an additive functor between semisimple categories.
\end{proof}

 The functor $\Gamma_{n}: \co^{\mathfrak{p}_{\infty}}_{\nu, \bC^{\infty}} \rightarrow \co^{\mathfrak{p}_{n}}_{\nu, \bC^{n}} $ clearly takes polynomial $\gl_{\infty}$-modules to polynomial $\gl_n$-modules; together with Lemma \ref{lem:Gamma_exact}, this means that $\Gamma_{n}$ factors through an exact functor $$\widehat{\Gamma}_{n}: \widehat{\co}^{\mathfrak{p}_{\infty}}_{\nu, \bC^{\infty}} \rightarrow \widehat{\co}^{\mathfrak{p}_{n}}_{\nu, \bC^{n}} $$
 i.e. we have a commutative diagram
 $$ \begin{CD}
     \co^{\mathfrak{p}_{\infty}}_{\nu, \bC^{\infty}} @>{\Gamma_{n}}>> \co^{\mathfrak{p}_{n}}_{\nu, \bC^{n}}\\
     @V{\hat{\pi}_{\infty}}VV @V{\hat{\pi}_{n}}VV \\
     \widehat{\co}^{\mathfrak{p}_{\infty}}_{\nu, \bC^{\infty}} @>{\widehat{\Gamma}_{n}}>> \widehat{\co}^{\mathfrak{p}_{n}}_{\nu, \bC^{n}}
    \end{CD}$$

\subsection{Stable inverse limit of \texorpdfstring{categories $\co^{\mathfrak{p}_{n}}_{\nu, \bC^{n}}$}{parabolic categories O} and the \texorpdfstring{category $\co^{\mathfrak{p}_{\infty}}_{\nu, \bC^{\infty}}$}{parabolic category O for infinite Lie algebra gl}}\label{ssec:inv_lim_par_cat_O}

The restriction functors $$\InnaD{\mathfrak{Res}}_{n-1, n}: \co^{\mathfrak{p}_{n}}_{\nu, \bC^{n}} \longrightarrow \co^{\mathfrak{p}_{n-1}}_{\nu, \bC^{n-1}}, \;\; \InnaD{\mathfrak{Res}}_{n-1, n} := ( \cdot)^{\gl_{n-1}^{\perp}}$$ descibed in Subsection \ref{ssec:Res_functors} allow us to consider the inverse limit of the system $((\co^{\mathfrak{p}_{n}}_{\nu, \bC^{n}})_{n \geq 1}, (\InnaD{\mathfrak{Res}}_{n-1, n})_{n \geq 2})$. 

Similarly, we can consider the inverse limit of the system $((\widehat{\co}^{\mathfrak{p}_{n}}_{\nu, \bC^{n}})_{n \geq 1}, (\widehat{\InnaD{\mathfrak{Res}}}_{n-1, n})_{n \geq 2})$. 

Let $n \InnaD{\geq 1}$.  
\begin{notation}
 For each $k \in \bZ_+$, let $Fil_k({\co}^{\mathfrak{p}_n}_{\nu, \bC^n})$ (respectively, $Fil_k(\widehat{\co}^{\mathfrak{p}_n}_{\nu, \bC^n})$) be the Serre subcategory of ${\co}^{\mathfrak{p}_n}_{\nu, \bC^n}$ (respectively, $\widehat{\co}^{\mathfrak{p}_n}_{\nu, \bC^n}$) generated by simple modules $L_n(\nu- \abs{\lambda}, \lambda)$ (respectively, $\hat{\pi}_n(L_n(\nu- \abs{\lambda}, \lambda))$), with $\ell(\lambda) \leq k$.
\end{notation}

\InnaD{
This defines $\bZ_+$-filtrations on \InnaD{the objects of} ${\co}^{\mathfrak{p}_n}_{\nu, \bC^n}$, $\widehat{\co}^{\mathfrak{p}_n}_{\nu, \bC^n}$, i.e.
$${\co}^{\mathfrak{p}_n}_{\nu, \bC^n} \cong \varinjlim_{k \in \bZ_+} Fil_k({\co}^{\mathfrak{p}_n}_{\nu, \bC^n})  \, ,\; \;\widehat{\co}^{\mathfrak{p}_n}_{\nu, \bC^n} \cong \varinjlim_{k \in \bZ_+} Fil_k(\widehat{\co}^{\mathfrak{p}_n}_{\nu, \bC^n}) $$}
\begin{lemma}
 Let $n \geq 1$. The functors \InnaD{$${\InnaD{\mathfrak{Res}}}_{n-1, n}: {\co}^{\mathfrak{p}_n}_{\nu, \bC^n} \longrightarrow {\co}^{\mathfrak{p}_{n-1}}_{\nu, \bC^{n-1}}$$
 and
 $$\widehat{\InnaD{\mathfrak{Res}}}_{n-1, n}:\widehat{\co}^{\mathfrak{p}_n}_{\nu, \bC^n} \longrightarrow \widehat{\co}^{\mathfrak{p}_{n-1}}_{\nu, \bC^{n-1}}$$ are both shortening and $\bZ_+$-filtered functors between finite-length categories with $\bZ_+$-filtrations on objects (see \cite{EA1} for the relevant definitions).}
\end{lemma}
\begin{proof}
\InnaD{These statements follow directly from Lemma \ref{lem:Res_image_simples}, which tells us that} 

$\InnaD{\mathfrak{Res}}_{n-1, n}(L_n(\nu- \abs{\lambda}, \lambda)) \cong L_{n-1}(\nu- \abs{\lambda}, \lambda)$.
\end{proof}

%
%

We can now consider the inverse \InnaD{limits} of the $\bZ_+$-filtered systems \InnaD{$(({\co}^{\mathfrak{p}_n}_{\nu, \bC^n})_{\InnaD{n \geq 1}}, ({\InnaD{\mathfrak{Res}}}_{n-1, n})_{n \geq \InnaD{2}})$ and} $((\widehat{\co}^{\mathfrak{p}_n}_{\nu, \bC^n})_{\InnaD{n \geq 1}}, (\widehat{\InnaD{\mathfrak{Res}}}_{n-1, n})_{n \geq \InnaD{2}})$. \InnaE{By \cite[Section 6]{EA1}, these} limits are \InnaE{equivalent to the respective} restricted invverse limits $$\varprojlim_{\InnaD{n \geq 1}, \text{ restr}} {\co}^{\mathfrak{p}_n}_{\nu, \bC^n}, \; \; \;  \varprojlim_{\InnaD{n \geq 1}, \text{ restr}} \widehat{\co}^{\mathfrak{p}_n}_{\nu, \bC^n}$$ 

The functors $\Gamma_n$ described above induce exact \InnaD{functors}
 $$\Gamma_{\text{lim}}: \co^{\mathfrak{p}_{\infty}}_{\nu, \bC^{\infty}} \longrightarrow \varprojlim_{\InnaD{n \geq 1}} \co^{\mathfrak{p}_{n}}_{\nu, \bC^{n}}$$ 
 \InnaD{and
 $$\widehat{\Gamma}_{\text{lim}}: \widehat{\co}^{\mathfrak{p}_{\infty}}_{\nu, \bC^{\infty}} \longrightarrow \varprojlim_{\InnaD{n \geq 1}} \widehat{\co}^{\mathfrak{p}_{n}}_{\nu, \bC^{n}}$$ }
 
 We would like to show that this functor is an equivalence of categories:
\begin{proposition}\label{prop:inv_lim_gl_infty_par_cat}
 The functors $\Gamma_n$ induce an equivalence
 $$\Gamma_{\text{lim}}: \co^{\mathfrak{p}_{\infty}}_{\nu, \bC^{\infty}} \longrightarrow \varprojlim_{\InnaD{n \geq 1}, \text{ restr}} \co^{\mathfrak{p}_{n}}_{\nu, \bC^{n}}$$ 
\end{proposition}
\begin{proof}
First of all, we need to \InnaD{check that this functor is well-defined. Namely, we need to show} that for any $M \in \co^{\mathfrak{p}_{\infty}}_{\nu, \bC^{\infty}}$, the sequence $\{ \ell_{\mathcal{U}(\gl_{n+1})}(\Gamma_{n+1}(M)) \}_n$ stabilizes. In fact, it is enough to show that this sequence is bounded (since it is obviously increasing).
 
Recall that we have a surjective map of $Sym(\mathfrak{u}_{\mathfrak{p}_{\infty}}^{-})$-modules $Sym (\mathfrak{u}_{\mathfrak{p}_{\infty}}^{-}) \otimes E \twoheadrightarrow M$ for some $E \in Rep(\gl(U_{\infty}))_{poly}$. Since $\Gamma_{n+1}$ is exact, it gives us a surjective map $Sym (\mathfrak{u}_{\mathfrak{p}_{n+1}}^{-}) \otimes \Gamma_{n+1}(E) \twoheadrightarrow \Gamma_{n+1}(M)$ for any $n \geq 0$, with $\Gamma_{n+1}(E)$ being a polynomial $\gl(U_{n+1})$-module.
 
 Now, $$  \ell_{\mathcal{U}(\gl_{n+1})}(\Gamma_{n+1}(M)) \leq \ell_{\mathcal{U}(\mathfrak{u}_{\mathfrak{p}_{n+1}}^{-})} (\Gamma_{n+1}(M)) \leq \ell_{\mathcal{U}(\gl(U_{n+1}))} (\Gamma_{n+1}(E))$$
 The sequence $\{ \ell_{\mathcal{U}(\gl(U_{n+1}))} (\Gamma_{n+1}(E)) \}_{n \geq 0}$ is bounded by Proposition \ref{prop:inv_lim_cat_poly_rep}, and thus the sequence $\{ \ell_{\mathcal{U}(\gl_{n+1})}(\Gamma_{n+1}(M)) \}_n$ is bounded as well. 
 
 \mbox{}
 We now show that $\Gamma_{\text{lim}}$ is an equivalence.
 
A construction similar to the one appearing in \InnaD{\cite[Section 7.5]{EA1}} gives a left-adjoint to the functor $\Gamma_{\text{lim}}$; that is, we will define a functor 
 $$\Gamma_{\text{lim}}^*: \varprojlim_{\InnaD{n \geq 1}, \text{ restr}} {\co}^{\mathfrak{p}_{n}}_{\nu, \bC^{n}}   \longrightarrow {\co}^{\mathfrak{p}_{\infty}}_{\nu, \bC^{\infty}} $$
 
 \InnaD{Let} $((M_n)_{n\geq 1}, (\phi_{n-1, n})_{n \geq 2})$ \InnaD{be an object} of $\varprojlim_{\InnaD{n \geq 1}, \text{ restr}} {\co}^{\mathfrak{p}_{n}}_{\nu, \bC^{n}}$. 
 
 \InnaD{The isomorphisms $\phi_{n-1, n} : \mathfrak{Res}_{n-1, n}(M_n) \stackrel{\sim}{\rightarrow} M_{n-1}$ define $\gl_{n-1}$-equivariant inclusions $M_{n-1} \hookrightarrow M_n$.} Consider \InnaD{the} vector space $$M := \bigcup_{n \geq 1} M_n$$ \InnaD{which has} a natural action of $\gl_{\infty}$ \InnaD{on it}.
 
 It is easy to see that the obtained $\gl_{\infty}$-module $M$ is a direct sum of polynomial $\gl(U_{\infty})$-modules, and is locally nilpotent over \InnaD{the algebra} $$\mathcal{U}(\mathfrak{u}_{\mathfrak{p}_{\infty}}^{+}) \cong Sym(U_{\infty, *}) \InnaD{\cong \bigcup_{n \geq 1} Sym(U_{n}^*)}$$ 
 
 \mbox{}
 \InnaD{We now prove the following sublemma:}
 \begin{sublemma}
  \InnaD{Let $((M_n)_{n\geq 1}, (\phi_{n-1, n})_{n \geq 2})$ \InnaD{be an object} of $\varprojlim_{\InnaD{n \geq 1}, \text{ restr}} {\co}^{\mathfrak{p}_{n}}_{\nu, \bC^{n}}$. Then $M := \bigcup_{n \geq 1} M_n$ is a finitely generated module over $Sym(U_{\infty}) \cong \mathcal{U}(\mathfrak{u}_{\mathfrak{p}_{\infty}}^{-})$. }
 \end{sublemma}
\begin{proof}
 \InnaD{Recall from \cite[Section 4]{EA1} that all the objects in the abelian category $\varprojlim_{\InnaD{n \geq 1}, \text{ restr}} {\co}^{\mathfrak{p}_{n}}_{\nu, \bC^{n}}$ have finite length, and that the simple objects in this category are exactly those of the form $((L_n(\nu- \abs{\lam}, \lam))_{n\geq 1}, (\phi_{n-1, n})_{n \geq 2})$ for a fixed Young diagram $\lam$. So we only need to check that applying the above construction to these simple objects gives rise to finitely generated modules over $Sym(U_{\infty}) \cong \mathcal{U}(\mathfrak{u}_{\mathfrak{p}_{\infty}}^{-})$. 
 
 Using \InnaE{Corollary \ref{cor:simple_Verma_basis_Groth_group}} we now reduce the proof of the sublemma to proving the following two statements:
 
 \begin{itemize}
  \item \InnaE{Let $\lam$ be a fixed Young diagram and let} $((L_n(\nu- \abs{\lam}, \lam))_{n\geq 1}, (\phi_{n-1, n})_{n \geq 2})$ be a simple object in $\varprojlim_{\InnaD{n \geq 1}, \text{ restr}} {\co}^{\mathfrak{p}_{n}}_{\nu, \bC^{n}}$ in which $L_n(\nu- \abs{\lam}, \lam)$ is polynomial for every $n$ (i.e. $\lam$ is minimal in its non-trivial $\stackrel{\nu}{\sim}$-class). 
  
  Then $L := \bigcup_{n \geq 1} L_n(\nu- \abs{\lam}, \lam)$ is a polynomial $\gl_{\infty}$-module (in particular, a finitely generated module over $Sym(U_{\infty}) \cong \mathcal{U}(\mathfrak{u}_{\mathfrak{p}_{\infty}}^{-})$).
  
  \item Let $\lam$ be a fixed Young diagram and let $(( M_{\mathfrak{p}_n}(\nu- \abs{\lam}, \lam) )_{n\geq 1}, (\phi_{n-1, n})_{n \geq 2})$ be an object of $\varprojlim_{\InnaD{n \geq 1}, \text{ restr}} {\co}^{\mathfrak{p}_{n}}_{\nu, \bC^{n}}$ (this is a sequence of ``compatible'' parabolic Verma modules). Then $$M := \bigcup_n M_{\mathfrak{p}_n}(\nu- \abs{\lam}, \lam)$$ is a finitely generated module over $Sym(U_{\infty}) \cong \mathcal{U}(\mathfrak{u}_{\mathfrak{p}_{\infty}}^{-})$.
 \end{itemize}
 \InnaE{The first statement follows immediately from} Proposition \ref{prop:inv_lim_cat_poly_rep} \InnaE{(cf. \cite[Section 7.5]{EA1})}. 
 
 To prove the second statement, recall that $$M_{\mathfrak{p}_n}(\nu- \abs{\lam}, \lam) \cong Sym(U_n) \otimes S^{\lambda} U_n$$ (Lemma \ref{lem:gl_u_struct_o_cat}). So $$M := \bigcup_n M_{\mathfrak{p}_n}(\nu- \abs{\lam}, \lam) \cong \bigcup_n Sym(U_n) \otimes S^{\lambda} U_n \cong Sym(U_{\infty}) \otimes S^{\lambda} U_{\infty}$$ which is clearly a finitely generated module over $Sym(U_{\infty}) \cong \mathcal{U}(\mathfrak{u}_{\mathfrak{p}_{\infty}}^{-})$.
 }
 \end{proof}
 
 This allows us to define the functor \InnaD{$\Gamma_{\text{lim}}^*$ by setting 
 $$\Gamma_{\text{lim}}^*( \, (M_n)_{n\geq 1}, \, (\phi_{n-1, n})_{n \geq 2} \, ) := \bigcup_{n \geq 1} M_n$$ and requiring that it act on morphisms accordingly.}
 
 The definition of $ \Gamma_{\text{lim}}^*$ gives us a natural transformation $$ \Gamma_{\text{lim}}^* \circ \Gamma_{\text{lim}} \stackrel{\sim}{\longrightarrow} \id_{{\co}^{\mathfrak{p}_{\infty}}_{\nu, \bC^{\infty}}} $$ Restricting the action of $\gl_{\infty}$ to $\gl(U_{\infty})$ and using Proposition \ref{prop:inv_lim_cat_poly_rep}, we conclude that this natural transformation is an isomorphism.
 
 Notice that the definition of $ \Gamma_{\text{lim}}^*$ implies that this functor is faithful. Thus we conclude that the functor $ \Gamma_{\text{lim}}^*$ is an equivalence of categories, and so is $ \Gamma_{\text{lim}}$.
 
\end{proof}
\begin{proposition}\label{prop:inv_lim_hat_gl_infty_par_cat}
 The functors $\widehat{\Gamma}_n$ induce an equivalence
 $$\widehat{\Gamma}_{\text{lim}}: \widehat{\co}^{\mathfrak{p}_{\infty}}_{\nu, \bC^{\infty}} \rightarrow \varprojlim_{\InnaD{n \geq 1}, \text{ restr}} \widehat{\co}^{\mathfrak{p}_{n}}_{\nu, \bC^{n}} $$ 
\end{proposition}
\begin{proof}
Let $M \in {\co}^{\mathfrak{p}_{\infty}}_{\nu, \bC^{\infty}}$. First of all, we need to \InnaD{check that the functor $\widehat{\Gamma}_{\text{lim}}$ is well-defined; that is, we need to} show that the sequence $\{ \ell_{\widehat{\co}^{\mathfrak{p}_{n}}_{\nu, \bC^{n}}}( \hat{\pi}_{n}(\Gamma_{n}(M))) \}_{n \geq 1}$ is bounded from above. 

Indeed, 
$$\ell_{\widehat{\co}^{\mathfrak{p}_{n}}_{\nu, \bC^{n}}}( \hat{\pi}_{n}(\Gamma_{n}(M))) \leq \ell_{{\co}^{\mathfrak{p}_{n}}_{\nu, \bC^{n}}}(\Gamma_{n}(M))$$
But the sequence $\{ \ell_{{\co}^{\mathfrak{p}_{n}}_{\nu, \bC^{n}}}( \Gamma_{n}(M)) \}_{n \geq 1}$ is bounded from above by Lemma \ref{prop:inv_lim_gl_infty_par_cat}, so the original sequence is bound from above as well.

\InnaD{Thus we obtain a commutative diagram
$$ \begin{CD}  Rep(\gl_{\infty})_{poly, \nu} @>>> {\co}^{\mathfrak{p}_{\infty}}_{\nu, \bC^{\infty}} @>{\hat{\pi}_{\infty}}>> \widehat{\co}^{\mathfrak{p}_{\infty}}_{\nu, \bC^{\infty}} \\
@V{\Gamma_{\text{lim}}}VV @V{\Gamma_{\text{lim}}}VV @V{\widehat{\Gamma}_{\text{lim}}}VV \\
\varprojlim_{\InnaD{n \geq 1}, \text{ restr}} Rep(\gl_{n})_{poly, \nu} @>>> \varprojlim_{\InnaD{n \geq 1}, \text{ restr}} {\co}^{\mathfrak{p}_{n}}_{\nu, \bC^{n}} @>{\hat{\pi}_{\text{lim}} = \varprojlim_n \hat{\pi}_n}>> \varprojlim_{\InnaD{n \geq 1}, \text{ restr}} \widehat{\co}^{\mathfrak{p}_{n}}_{\nu, \bC^{n}}  \end{CD}$$
where $Rep(\gl_{N})_{poly, \nu}$ is the Serre subcategory of $\widehat{\co}^{\mathfrak{p}_{N}}_{\nu, \bC^{N}}$ consisting of all polynomial modules of degree $\nu$. The rows of this commutative diagram are ``exact'' (in the sense that $\widehat{\co}^{\mathfrak{p}_{\infty}}_{\nu, \bC^{\infty}}$ is the Serre quotient of the category ${\co}^{\mathfrak{p}_{\infty}}_{\nu, \bC^{\infty}}$ by the Serre subcategory $Rep(\gl_{\infty})_{poly, \nu}$, and similarly for the bottom row). 

The functors $$\Gamma_{\text{lim}}: Rep(\gl_{\infty})_{poly, \nu} \longrightarrow \varprojlim_{\InnaD{n \geq 1}, \text{ restr}} Rep(\gl_{n})_{poly, \nu}$$ and $$\Gamma_{\text{lim}}: {\co}^{\mathfrak{p}_{\infty}}_{\nu, \bC^{\infty}} \longrightarrow \varprojlim_{\InnaD{n \geq 1}, \text{ restr}} {\co}^{\mathfrak{p}_{n}}_{\nu, \bC^{n}}$$ are equivalences of categories (by Propositions \ref{prop:inv_lim_cat_poly_rep} and \ref{prop:inv_lim_gl_infty_par_cat}), and thus the functor $\widehat{\Gamma}_{\text{lim}}$ is an equivalence as well.
}

%
%
%
\end{proof}

\section{Complex tensor powers of a unital vector space}\label{sec:comp_tens_power}

In this section we describe the construction of a complex tensor power of the unital vector space $\bC^N$ with the chosen vector $\triv:=e_1$ (again, $N \in \bZ_+ \cup \{\infty \}$). A general construction of the complex tensor power of a unital vector space is given in \cite[Section 6]{EA}. 

Again, we denote $U_N:= span\{e_2, e_3, ...\}$, and $U_{N *}:= span \{e_2^*, e_3^*, ...\} \subset \bC^N_*$. \InnaD{As before, we have a decomposition:
$$\gl_N \cong \bC   \id_{\bC^N} \oplus \mathfrak{u}_{\mathfrak{p}_N}^{-} \oplus \mathfrak{u}_{\mathfrak{p}_N}^{+}  \oplus \gl(U_N)$$ such that} $U_N \cong \mathfrak{u}_{\mathfrak{p}_N}^{-}$, $U_{N *} \cong \mathfrak{u}_{\mathfrak{p}_N}^{+}$, and if $N$ is finite, we have $U_N^* \cong U_{N *}$.

Fix $\nu \in \bC$.

\InnaD{Recall from \cite[Section 4]{EA} that for any $\nu \in \bC$, in the Deligne category $\underline{Rep}(S_{\nu})$ we have the objects $\Delta_k$ ($k \in \bZ_+$). These objects interpolate the representations $\bC Inj(\{1,...,k\} , \{1,...,n\}) \cong Ind_{S_{n-k} \times S_k \times S_k}^{S_n \times S_k} \bC$ of the symmetric groups $S_n$; in fact, for any $n \in \bZ_+$ we have: $$\mathcal{S}_n(\Delta_k) \cong \bC Inj(\{1,...,k\} , \{1,...,n\})$$ where $\mathcal{S}_n: \underline{Rep}(S_{\nu=n}) \rightarrow Rep(S_n)$ is the monoidal functor discussed in Subsection \ref{ssec:S_nu_general}.}
%
%
\begin{definition}[Complex tensor power]\label{def:complex_ten_power_splitting}
 
 Define the object $(\bC^N)^{\underline{\otimes}  \nu}$ of $Ind-(\underline{Rep}^{ab}(S_{\nu}) \boxtimes \co^{\mathfrak{p}_N}_{\nu, \bC^N})$ \InnaA{by setting}
$$ (\bC^N)^{\underline{\otimes}  \nu} := \bigoplus_{k \geq 0} (U_N^{\otimes k} \otimes \Del_k)^{S_k}$$

The action on $\mathfrak{gl}_N$ on $(\bC^N)^{\underline{\otimes}  \nu}$ is given as follows:

\InnaA{$$ \entrymodifiers={+!!<0pt,\fontdimen22\textfont2>} \xymatrix{ &{\mathbf{1}\;} \ar@/^1pc/[r]^{U_N} &{\phantom{UUU} {U_N} \otimes \Del_1 \phantom{UU}}  \ar@/^1pc/[r]^{U_N} \ar@/^1pc/[l]^{U_{N *}} \ar@(dl,dr)_{\gl({U_N})} &{\,\phantom{UU} (U_N^{\otimes 2} \otimes \Del_2)^{S_2} \phantom{UU}} \ar@/^1pc/[r]^{U_N} \ar@/^1pc/[l]^{U_{N *}} \ar@(dl,dr)_{\gl(U_N)} &{\,\phantom{UU} (U_N^{\otimes 3} \otimes \Del_3)^{S_3} \phantom{UUU}} \ar@/^1pc/[r]^-{U_N} \ar@/^1pc/[l]^{U_{N *}} \ar@(dl,dr)_{\gl(U_N)} &{\,\phantom{U} \ldots \,} \ar@/^1pc/[l]^{U_{N *}} }$$}

\begin{itemize}
 \item $E_{1,1} \in \mathfrak{gl}_N$ acts by scalar $\nu - k$ on each summand $(U_N^{\otimes k} \otimes \Del_k)^{S_k} $.

\item $A \in \mathfrak{gl}(U_N) \subset \mathfrak{gl}_N$ acts on $(U_N^{\otimes k} \otimes \Del_k)^{S_k} $ by

\begin{align*}
\sum_{1 \leq i \leq k} A^{(i)} \rvert_{U_N^{\otimes k}} \otimes \id_{\Del_k}: (U_N^{\otimes k} \otimes \Del_k)^{S_k}  &\longrightarrow (U_N^{\otimes k} \otimes \Del_k)^{S_k}
\end{align*}
\item $u \in  U_N \cong \mathfrak{u}_{\mathfrak{p}_N}^{-}$ acts by morphisms of degree $1$, which are given explicitly in \cite[Section 6.2]{EA}.

\item $f \in  U_{N *} \cong \mathfrak{u}_{\mathfrak{p}_N}^{+}$ acts by morphisms of degree $-1$, which are given explicitly in \cite[Section 6.2]{EA}.
\end{itemize}
\end{definition}

\begin{remark}\label{rmk:uniq_determ_actions_of_u_plus_minus}
 \InnaD{The actions of the elements of $\mathfrak{u}_{\mathfrak{p}_N}^{+}$, $\mathfrak{u}_{\mathfrak{p}_N}^{-}$, though not written here explicitly, are in fact uniquely determined by the actions of $E_{1, 1}$ and $\gl(U_N)$.
 
 To see this, note that the ideal in the Lie algebra $\gl_{N}$ generated by the Lie subalgebra $\bC E_{1, 1} \oplus \gl(U_{N})$ is the \InnaE{entire} $ \gl_{N}$. Given two $\gl_{N}$-modules $M_1, M_2$ and an isomorphism $M_1 \rightarrow M_2$ which is equivariant with respect to the Lie subalgebra $\bC E_{1, 1} \oplus \gl(U_{N})$, \InnaE{the above fact implies that} this isomorphism is also \InnaE{$\gl_{N}$-}equivariant. 

 In other words, if there exists a way to define an action of $\gl_N$ whose restriction to the the Lie subalgebra $\bC E_{1, 1} \oplus \gl(U_{N})$ is given by the formulas above, then such an action of $\gl_N$ is unique.
 }
\end{remark}
\begin{remark}\label{rmk:comp_tens_power_O_cat}
\InnaD{The proof that the object $(\bC^N)^{\underline{\otimes}  \nu}$ lies in the category $Ind-(\underline{Rep}(S_{\nu}) \InnaB{\boxtimes} \co^{\mathfrak{p}_N}_{\nu, \bC^N}$ is an easy check, and can be found in \cite{EA}. In particular, it means that the action of the mirabolic subalgebra $\InnaE{\operatorname{Lie} \bar{\mathfrak{P}}_{\triv}}$ on the complex tensor power $(\bC^N)^{\underline{\otimes}  \nu}$ integrates to an action of the mirabolic subgroup $\bar{\mathfrak{P}}_{\triv}$, thus making $(\bC^N)^{\underline{\otimes}  \nu}$ a Harish-Chandra module in $Ind-\underline{Rep}^{ab}(S_{\nu})$ for the pair $(\gl_N, \bar{\mathfrak{P}}_{\triv})$.}
\end{remark}

The definition of \InnaD{the} complex tensor power is compatible with the usual notion of a tensor power of a unital vector space (see \cite[Section 6]{EA}):
\begin{proposition}\label{prop:comp_tens_power_F_n}
 Let $d \in \bZ_+$. Consider the functor $$\hat{\mathcal{S}}_d: Ind-(\underline{Rep}(S_{\nu=d}) \InnaB{\boxtimes} \co^{\mathfrak{p}_N}_{d, \bC^N}) \longrightarrow Ind-(Rep(S_d)\InnaB{\boxtimes} \co^{\mathfrak{p}_N}_{d, \bC^N})$$ induced by the functor $${\mathcal{S}}_d: \underline{Rep}(S_{\nu=d}) \longrightarrow Rep(S_n)$$ described in Subsection \ref{ssec:S_nu_general}. Then $\hat{\mathcal{S}}_d ((\bC^N)^{\underline{\otimes}  d}) \cong (\bC^N)^{\otimes d}$.
\end{proposition}

\InnaC{The construction of the complex tensor power is also compatible with the functors \InnaD{$\mathfrak{Res}_{n, n+1}$ and} $\Gamma_n$ defined in Definitions \ref{def:Res_funct_cat_o}, \ref{def:Gamma_funct_cat_o}. These properties can be seen as special cases of the following statement (when $N=n+1$ and $N=\infty$, respectively):}

\begin{proposition}\label{prop:compatab_comp_tens_power}
\InnaD{Let $n \geq 1$, and let $N \geq n$, $N \in \bZ_{\geq 1} \cup \{\infty \}$. Recall that we have an inclusion $\gl_n \oplus \gl_n^{\perp} \subset \gl_N$, and consider the functor
$$ (\cdot )^{\gl_n^{\perp}}: Ind-\left(\underline{Rep}^{ab}(S_{\nu}) \boxtimes {\co}^{\mathfrak{p}_{N}}_{\nu, \bC^{N}}\right) \longrightarrow Ind-\left(\underline{Rep}^{ab}(S_{\nu}) \boxtimes {\co}^{\mathfrak{p}_{n}}_{\nu, \bC^{n}} \right)$$ induced by the functor $(\cdot )^{\gl_n^{\perp}}:  {\co}^{\mathfrak{p}_{N}}_{\nu, \bC^{N}} \rightarrow {\co}^{\mathfrak{p}_{n}}_{\nu, \bC^{n}}$. The functor $ (\cdot )^{\gl_n^{\perp}}$ then takes $(\bC^{N})^{\underline{\otimes} \nu}$ to $(\bC^n)^{\underline{\otimes} \nu}$.}
\end{proposition}
\begin{proof}
\InnaD{The functor $(\cdot )^{\gl_n^{\perp}}:  {\co}^{\mathfrak{p}_{N}}_{\nu, \bC^{N}} \rightarrow {\co}^{\mathfrak{p}_{n}}_{\nu, \bC^{n}}$ induces an endofunctor of $\InnaD{Ind-}\underline{Rep}^{ab}(S_{\nu})$. \InnaD{We} would like to say that \InnaD{we have an isomorphism of $\InnaD{Ind-}\underline{Rep}^{ab}(S_{\nu})$-objects} 

$((\bC^{N})^{\underline{\otimes} \nu})^{\gl_n^{\perp}} \stackrel{?}{\cong} (\bC^n)^{\underline{\otimes} \nu}$, and \InnaD{that} the action of $\gl_n \InnaD{\subset \gl_{N}}$ on \InnaD{$((\bC^{N})^{\underline{\otimes} \nu})$} corresponds to the action of $\gl_n$ on \InnaD{$(\bC^n)^{\underline{\otimes} \nu}$}.

In order to do this, we first consider $(\bC^N)^{\underline{\otimes}  \nu}$ as \InnaA{an object in $Ind-\underline{Rep}^{ab}(S_{\nu})$ with an action of $\gl(U_N)$}:
$$(\bC^N)^{\underline{\otimes}  \nu} \cong \bigoplus_{k \geq 0} (\Del_k \otimes U^{\otimes k}_N)^{S_k}$$

If we consider only the actions of $\gl(U_{N}), \gl(U_n)$, the functor $\Gamma_n$ is induced by the \InnaD{additive} monoidal functor $(\cdot)^{\gl(U_n)^{\perp}}: Ind-Rep(\gl(U_{N}))_{poly} \rightarrow Ind-Rep(\gl(U_{N}))_{poly}$. \InnaD{This shows that we have an isomorphism of $\InnaD{Ind-}\underline{Rep}^{ab}(S_{\nu})$-objects $$((\bC^{N})^{\underline{\otimes} \nu})^{\gl_n^{\perp}}  \cong \bigoplus_{k \geq 0} (\Del_k \otimes U^{\otimes k}_n)^{S_k} \cong (\bC^n)^{\underline{\otimes} \nu}$$} and the actions of $\gl(U_n)$ on \InnaD{both sides are compatible. From the definition of the complex tensor power (Definition \ref{def:complex_ten_power_splitting})} one immediately sees that the actions of $E_{1,1}$ on both sides are compatible as well. Remark \ref{rmk:uniq_determ_actions_of_u_plus_minus} now completes the proof.

}
\end{proof}

\section{Schur-Weyl duality in complex rank: the Schur-Weyl functor and the finite-dimensional case}\label{sec:SW_functor}

We fix $\nu \in \bC$, and $N \in \bZ_+ \cup \{\infty\}$. Again, we consider the unital vector space $\bC^N$ with the chosen vector $\triv := e_1$ and the complement $U_N := span\{e_2, e_3, ...\}$. 
\subsection{Schur-Weyl functor}\label{ssec:SW_functor}
\begin{definition}\label{def:SW_functor}
 Define the Schur-Weyl contravariant functor $$SW_{\nu}: \underline{Rep}^{ab}(S_{\nu}) \longrightarrow Mod_{\mathcal{U}(\gl_N)}$$ by $$SW_{\nu}:= \Hom_{\underline{Rep}^{ab}(S_{\nu})}( \, \cdot \, , (\bC^N)^{\underline{\otimes}  \nu})$$

\end{definition}

\begin{remark}
 The functor $SW_{\nu}: \underline{Rep}^{ab}(S_{\nu}) \longrightarrow Mod_{\mathcal{U}(\gl_N)}$ is a contravariant $\bC$-linear additive left-exact functor.
\end{remark}


It turns out that the image of the functor $SW_{\nu} :\underline{Rep}^{ab}(S_{\nu}) \rightarrow Mod_{\mathcal{U}(\gl_N)}$ lies in $\co^{\mathfrak{p}_N}_{\nu, \bC^N}$ (cf. Remark \ref{rmk:comp_tens_power_O_cat}).

We can now define another Schur-Weyl functor which we will consider: it is the contravariant functor $\widehat{SW}_{\nu, {\bC^N}}: \underline{Rep}^{ab}(S_{\nu}) \longrightarrow \widehat{\co}^{\mathfrak{p}_N}_{\nu, {\bC^N}}$. \InnaD{Recall from Section \ref{ssec:def_par_cat_O_gl_N} that} $$\hat{\pi}_N: \co^{\mathfrak{p}_N}_{\nu, {\bC^N}} \longrightarrow \widehat{\co}^{\mathfrak{p}_N}_{\nu, {\bC^N}}:= \quotient{\co^{\mathfrak{p}_N}_{\nu, {\bC^N}}}{\InnaA{Rep(\gl_N)_{poly, \nu}}}$$ 
is the Serre quotient of $\co^{\mathfrak{p}_N}_{\nu, {\bC^N}}$ by the Serre subcategory of polynomial $\gl_N$-modules of degree $\nu$. 
We then define $$\widehat{SW}_{\nu, {\bC^N}} := \hat{\pi}_N \circ SW_{\nu, \bC^N}$$

\subsection{The finite-dimensional case}\label{ssec:finite_dim_SW}
Let $n \in \bZ_+$. We then have the following theorem, which can be found in \cite[Section 7]{EA}:

\begin{theorem}\label{thrm:SW_almost_equiv}
The contravariant functor $\widehat{SW}_{\nu, {\bC^n}}:\underline{Rep}^{ab}(S_{\nu}) \rightarrow \widehat{\co}^{\mathfrak{p}_n}_{\nu, {\bC^n}}$ is exact and essentially surjective.
Moreover, the induced contravariant functor $$ \quotient{\underline{Rep}^{ab}(S_{\nu})}{Ker(\widehat{SW}_{\nu, {\bC^n}})} \rightarrow \widehat{\co}^{\mathfrak{p}_n}_{\nu, {\bC^n}}$$ is an anti-equivalence of abelian categories, thus making $\widehat{\co}^{\mathfrak{p}_n}_{\nu, {\bC^n}}$ a Serre quotient of $\underline{Rep}^{ab}(S_{\nu})^{op}$.

\end{theorem}

We will show that a similar result holds \InnaD{in the infinite-dimensional case}, \InnaE{when} the contravariant functor $\widehat{SW}_{\nu, \bC^{\infty}}$ is \InnaE{in fact} an anti-equivalence of categories.

\InnaD{In the} proof of Theorem \ref{thrm:SW_almost_equiv} \InnaD{we established} the following fact (see \cite[Theorem 7.2.3]{EA}):

\begin{lemma}\label{lem:SW_images_simples}
The functor $\widehat{SW}_{\nu, \bC^n}$ takes a simple object to either a simple object, or zero. More specifically, we have:
 \begin{itemize}
  \item Let $\lambda$ be a Young diagram lying in a trivial $\stackrel{\nu}{\sim}$-class. Then 
  $$\widehat{SW}_{\nu, \bC^n}(\mathbf{L}(\lambda)) \cong \hat{\pi}(L_{\mathfrak{p}_n}(\nu-\abs{\lambda}, \lambda))$$
  \item Consider a non-trivial $\stackrel{\nu}{\sim}$-class $\{\lambda^{(i)} \}_{i \geq 0}$. Then $$\widehat{SW}_{\nu, \bC^n}(\mathbf{L}(\lambda^{(i)})) \cong \hat{\pi}(L_{\mathfrak{p}_n}(\nu-\abs{\lambda^{(i+1)}}, \lambda^{(i+1)}))$$ whenever $i\geq 0$.
 \end{itemize}
\end{lemma}

\begin{remark}
 Recall that $L_{\mathfrak{p}_n}(\nu-\abs{\lambda}, \lambda)$ is zero if $\ell(\lambda) \geq n$.
\end{remark}

\section{Classical Schur-Weyl duality and the restricted inverse limit}\label{sec:class_SW_inv_limit}

\subsection{\InnaD{A short overview of the} classical Schur-Weyl duality}\label{ssec:class_Schur_Weyl}

Let $V$ be a vector space over $\bC$, and let $\InnaD{d \in \bZ_+}$. The \InnaD{symmetric group} $S_d$ acts on $V^{\otimes d}$ by permuting the factors of the tensor product (the action is semisimple, by Mashke's theorem):\InnaA{
$$\sigma.(v_1 \otimes v_2 \otimes...\otimes v_d) := v_{\sigma^{-1}(1)} \otimes v_{\sigma^{-1}(2)} \otimes...\otimes v_{\sigma^{-1}(d)}$$}

The actions of $\gl(V)$, $S_d$ on $V^{\otimes d}$ commute, which allows us to consider a contravariant functor $$\mathtt{SW}_{d, V}: Rep(S_d) \rightarrow Rep(\gl(V))_{poly}, \, \mathtt{SW}_{d, V}:=\Hom_{S_d}(\cdot,  V^{\otimes d})$$

The contravariant functor $\mathtt{SW}_{d, V}$ is $\bC$-linear and additive, and sends a simple \InnaD{representation} $\lambda$ of $S_d$ to \InnaD{the $\gl(V)$-module} $S^{\lambda}V$.

Next, consider the contravariant functor $$\mathtt{SW}_V: \bigoplus_{d \in \bZ_+} Rep(S_d) \rightarrow Rep(\gl(V))_{poly}, \, \mathtt{SW}_V:= \oplus_d \mathtt{SW}_{d, V}$$
This functor $\mathtt{SW}_V$ is clearly essentially surjective and full (this is easy to see, since $Rep(\gl(V))_{poly}$ is a semisimple category with simple objects $\InnaA{S^{\lambda} V \cong} \mathtt{SW}(\lambda)$).

The kernel of the functor $\mathtt{SW}_V$ is the full additive subcategory (direct factor) of $\bigoplus_{d \in \bZ_+} Rep(S_d)$ generated by simple objects $\lambda$ such that $\ell(\lambda) > \dim V $.

\subsection{Classical Schur-Weyl duality: inverse limit}\label{ssec:class_Schur_Weyl_limit}
In this subsection, we prove that the classical Schur-Weyl functors $ \mathtt{SW}_{\bC^n}$ make the category $\bigoplus_{d \in \bZ_+} Rep(S_d)$ dual (anti-equivalent) to the category $$Rep(\gl_{\infty})_{poly} \cong \varprojlim_{n \in \bZ_+, \text{ restr}}  Rep(\gl_{n})_{poly} $$




The contravariant functor $\mathtt{SW}_{\bC^N}$ sends the Young diagram $\lambda$ to the $\gl_N$-module $S^{\lambda} \bC^N$.

%

Let $n \in \bZ_+$.
We start by noticing that the functors $\InnaD{\mathfrak{Res}}_{\InnaE{n, n+1}}$ and the functors $\Gamma_n$ (\InnaD{defined in} Section \ref{sec:poly_rep}) are compatible with the classical Schur-Weyl functors $\mathtt{SW}_{\bC^n}$:
\begin{lemma}
 We have natural isomorphisms $$ \InnaD{\mathfrak{Res}}_{n, n+1} \circ \mathtt{SW}_{\bC^{n+1}} \cong \mathtt{SW}_{\bC^{n}}$$
 and $$ \Gamma_n \circ \mathtt{SW}_{\bC^{\infty}} \cong \mathtt{SW}_{\bC^{n}}$$
 for any $n \geq 0$.
\end{lemma}
\begin{proof}
 It is enough to check this on simple objects in $\bigoplus_{d \in \bZ_+} Rep(S_d)$, \InnaA{in which case the statement} follows directly from the definitions of $\InnaD{\mathfrak{Res}}_{n, n+1}$, $ \Gamma_n$ together with the fact that $\mathtt{SW}_{\bC^N}(\lambda) \cong S^{\lambda} \bC^N$ \InnaD{for any} $N \in \bZ_+ \cup \{\infty \}$.
\end{proof}

\InnaD{The above Lemma implies that we have a commutative diagram 
 $$\xymatrix{&{} &{} &Rep(\gl_{n})_{poly} \\ &\bigoplus_{d \in \bZ_+} Rep(S_d)^{op} \ar[rr]_{\mathtt{SW}_{\InnaD{\text{lim}}}}  \ar[rru]^{\mathtt{SW}_{\bC^n}} \ar[rrd]_{\mathtt{SW}_{\bC^{\infty}}} &{} &\varprojlim_{n \geq 1, \text{ restr}} Rep(\gl_{n})_{poly} \ar[u]_{\mathbf{Pr}_n} \\ &{} &{} &Rep(\gl_{\infty})_{poly} \ar[u]_{\Gamma_{\InnaD{\text{lim}}}} \ar@/_6pc/[uu]_{\Gamma_n} }$$
the functor $\Gamma_{\text{lim}}$ being an equivalence of categories (by Proposition \ref{prop:inv_lim_cat_poly_rep}), \InnaE{and $\mathbf{Pr}_n$ being the canonical projection functor.}
}
\begin{proposition}
 The contravariant \InnaD{functors} $$ \mathtt{SW}_{\infty}: \bigoplus_{d \in \bZ_+} Rep(S_d) \longrightarrow  Rep(\gl_{\infty})_{poly}$$ \InnaD{and $$ \mathtt{SW}_{\text{lim}}: \bigoplus_{d \in \bZ_+} Rep(S_d) \rightarrow \varprojlim_{n \in \bZ_+, \text{ restr}} Rep(\gl_n)_{poly}$$ are} anti-equivalences of \InnaD{semisimple} categories. 
\end{proposition}
\begin{proof}
As it was said in Subsection \ref{ssec:class_Schur_Weyl}, the functor $ \mathtt{SW}_{N}$ is full and essentially surjective for any $N$. In this case, the functor $ \mathtt{SW}_{\infty}$ is also faithful, since the simple object $\lam$ in $\bigoplus_{d \in \bZ_+} Rep(S_d)$ is taken by the functor $ \mathtt{SW}_{\infty}$ to the simple object $S^{\lambda} \bC^{\infty} \neq 0$. \InnaD{This proves that the contravariant functor $\mathtt{SW}_{\infty}$ is an anti-equivalence of categories. The commutative diagram above then implies that the contravariant functor $ \mathtt{SW}_{\text{lim}}$ is an anti-equivalence as well.}
\end{proof}

%

\section{\texorpdfstring{$\underline{Rep}^{ab}(S_{\nu})$}{Deligne category} and the inverse limit of \texorpdfstring{categories $\widehat{\co}^{\mathfrak{p}_N}_{\nu, \bC^N}$}{parabolic categories O}}\label{sec:SW_duality_limit}
\subsection{}
In this section we are going to prove that the Schur-Weyl functors defined in Section \ref{ssec:SW_functor} give us an equivalence of categories between $\Dab$ and \InnaD{the} \InnaA{restr} inverse limit $\varprojlim_{N \in \bZ_+, \text{ restr}} \widehat{\co}^{\mathfrak{p}_N}_{\nu, \bC^N}$.

\InnaA{We fix $\nu \in \bC$.}


\begin{proposition}\label{prop:Res_nu_n_compatible}
The functor $\InnaC{\InnaD{\mathfrak{Res}}}_{n-1, n}$ satisfies: $\InnaC{\InnaD{\mathfrak{Res}}}_{n-1, n} \circ SW_{\nu, \bC^n} \cong  SW_{\nu, \bC^{n-1}}$, i.e. there exists a natural isomorphism $\eta_n: \InnaC{\InnaD{\mathfrak{Res}}}_{n-1, n} \circ SW_{\nu, \bC^n} \rightarrow  SW_{\nu, \bC^{n-1}}$. 
\end{proposition}
\begin{proof}
\InnaD{Follows directly from Proposition \ref{prop:compatab_comp_tens_power}.}
\end{proof}

\begin{corollary}\label{cor:overline_Res_nu_n_compatible}
 We have $\widehat{\InnaD{\mathfrak{Res}}}_{n-1, n} \circ \widehat{SW}_{\nu, \bC^n} \cong \widehat{SW}_{\nu, \bC^{n-1}}$, i.e. there exists a natural isomorphism $\InnaA{\hat{\eta}}_n: \widehat{\InnaD{\mathfrak{Res}}}_{n-1, n} \circ \widehat{SW}_{\nu, \bC^n} \rightarrow \widehat{SW}_{\nu, \bC^{n-1}}$.
\end{corollary}

\begin{proof}
 By definition of $\widehat{\InnaD{\mathfrak{Res}}}_{n-1, n}, \widehat{SW}_{\nu, \bC^n}$, together with Proposition \ref{prop:Res_nu_n_compatible}, we have a commutative diagram
 $$ \xymatrix{ &\Dab^{\InnaE{op}} \ar[rr]_{{SW}_{\nu, \bC^n}} \ar@/^1.3pc/[rrdrru]^{{SW}_{\nu, \bC^{n-1}}} \ar[rrdd]_{\widehat{SW}_{\nu, \bC^n}} &{} &\co^{\mathfrak{p}_{n}}_{\nu, \bC^{n}} \ar[dd]^{\hat{\pi}_n} \ar[rr]_{\InnaC{\InnaD{\mathfrak{Res}}}_{n-1, n}} &{} &\co^{\mathfrak{p}_{n-1}}_{\nu, \bC^{n-1}} \ar[dd]^{\hat{\pi}_{n-1}} \\ &{} &{} &{} &{} &{} \\ &{} &{} &\widehat{\co}^{\mathfrak{p}_{n}}_{\nu, \bC^{n}} \ar[rr]_{\widehat{\InnaD{\mathfrak{Res}}}_{n-1, n}} &{} &\widehat{\co}^{\mathfrak{p}_{n-1}}_{\nu, \bC^{n-1}} } $$
 Since $\hat{\pi}_{n-1} \circ {SW}_{\nu, \bC^{n-1}} =:  \widehat{SW}_{\nu, \bC^{n-1}}$, we get $\widehat{\InnaD{\mathfrak{Res}}}_{n-1, n} \circ \widehat{SW}_{\nu, \bC^n} \cong \widehat{SW}_{\nu, \bC^{n-1}}$. 
\end{proof}

\mbox{}

\begin{notation}
 For each $k \in \bZ_+$, $Fil_k(\Dab)$ is defined to be the Serre subcategory \InnaD{of $\Dab$} generated by the simple objects $\mathbf{L}({\lambda})$ such that the Young diagram $\lambda$ satisfies either of the following conditions:
 \begin{itemize}
  \item $\lambda$ belongs to a trivial $\stackrel{\nu}{\sim}$-class, and $\ell(\lambda) \leq k$.
  \item $\lambda$ belongs to a non-trivial $\stackrel{\nu}{\sim}$-class $\{ \lambda^{(i)} \}_{i \geq 0}$, $\lambda = \lambda^{(i)}$, and $\ell(\lambda^{(i+1)}) \leq k$.
 \end{itemize}
\end{notation}

\InnaD{This defines a $\bZ_+$-filtration \InnaD{on the objects} of the category $\Dab$. That is, we have:
$$\Dab \cong \varinjlim_{k \in \bZ_+} Fil_k(\Dab)$$}

\begin{lemma}
 The functors $\widehat{SW}_{\nu, \bC^n}$ are $\bZ_+$-filtered shortening functors \InnaD{(see \cite{EA1} for the relevant definitions)}.
\end{lemma}
\begin{proof}
Follows from the fact that $\widehat{SW}_{\nu, \bC^n}$ are exact, together with Lemma \ref{lem:SW_images_simples}.
\end{proof}
\InnaD{
This Lemma, together with Corollary \ref{cor:overline_Res_nu_n_compatible}, gives us a contravariant ($\bZ_+$-filtered shortening) functor
 \begin{align*}
  \widehat{SW}_{\nu, \InnaD{\text{lim}}}: \Dab &\longrightarrow \varprojlim_{\InnaD{n \geq 1}, \text{ restr}} \widehat{\co}^{\mathfrak{p}_n}_{\nu, \bC^n} \\
  X &\mapsto \left( \{\widehat{SW}_{\nu,  \bC^n}(X)\}_{\InnaD{n \geq 1}}, \{\InnaA{\hat{\eta}}_{n} (X)\}_{n \geq 2} \right)\\
  (f:X \rightarrow Y) &\mapsto \{\widehat{SW}_{\nu,  \bC^n}(f): \widehat{SW}_{\nu,  \bC^n}(Y) \rightarrow \widehat{SW}_{\nu,  \bC^n}(X)\}_{\InnaD{n \geq 1}}
 \end{align*}
 This functor is given by the universal property of the restricted inverse limit described in \cite{EA1}, and makes the diagram below commutative: 
$$\xymatrix{&{} &{} &\widehat{\co}^{\mathfrak{p}_{n}}_{\nu, \bC^{n}} \\ &\underline{Rep}^{ab}(S_{\nu})^{op} \ar[rr]_{\widehat{SW}_{\nu, \InnaD{\text{lim}}}}  \ar[rru]^{\widehat{SW}_{\nu, \bC^n}}  &{} &\varprojlim_{n \geq 1, \text{ restr}} \widehat{\co}^{\mathfrak{p}_{n}}_{\nu, \bC^{n}} \ar[u]_{\mathbf{Pr}_n}  }$$
\InnaE{(here $\mathbf{Pr}_n$ is the canonical projection functor).}
}
%

We now show that there is an equivalence of categories $\Dab^{op}$ and $\varprojlim_{\InnaD{n \geq 1}, \text{ restr}} \widehat{\co}^{\mathfrak{p}_n}_{\nu, \bC^n}$.

\begin{theorem}\label{thrm:SW_limit_equivalence}
 The Schur-Weyl contravariant functors $\widehat{SW}_{\nu, \bC^n}$ induce an anti-equivalence of abelian categories, given by the \InnaE{(exact)} \InnaD{contravariant functor
$$  \widehat{SW}_{\nu, \InnaD{\text{lim}}}: \Dab \longrightarrow \varprojlim_{\InnaD{n \geq 1}, \text{ restr}} \widehat{\co}^{\mathfrak{p}_n}_{\nu, \bC^n}$$}
\end{theorem}
\begin{proof}
The functors $\widehat{SW}_{\nu, \bC^n}$ are exact for each $\InnaD{n \geq 1}$, which means (see \cite[Section 3.2]{EA1}) that the functor $\widehat{SW}_{\nu, \InnaD{\text{lim}}}$ is exact as well.

To see that it is an anti-equivalence, we will use \cite[Proposition 5.1.10]{EA1}. All we need to check is that the functors $\widehat{SW}_{\nu, \bC^n}$ satisfy the ``stabilization condition'' (\cite[Condition 5.1.9]{EA1}): that is, for each $k \in \bZ_+$, there exists $n_k \in \bZ_+$ such that $$\widehat{SW}_{\nu, \bC^n}: Fil_k(\Dab) \rightarrow Fil_k(\widehat{\co}^{\mathfrak{p}_n}_{\nu, \bC^n})$$ is an anti-equivalence of categories for any $n \geq n_k$.

Indeed, let $k \in \bZ_+$, and let $n \geq k +1$. 

The category $ Fil_k(\Dab)$ decomposes into blocks (corresponding to the blocks of $\Dab$), and the category $Fil_k(\widehat{\co}^{\mathfrak{p}_n}_{\nu, \bC^n})$ decomposes into blocks corresponding to the blocks of $\widehat{\co}^{\mathfrak{p}_n}_{\nu, \bC^n}$. 

The requirement $n \geq k +1$ together with Lemma \ref{lem:SW_images_simples} means that for any semisimple block of $ Fil_k(\Dab)$, the simple object $\mathbf{L}(\lambda)$ corresponding to this block is not sent to zero under $\widehat{SW}_{\nu, \bC^n}$. This, in turn, implies that $\widehat{SW}_{\nu, \bC^n}$ induces an anti-equivalence between each semisimple block of $ Fil_k(\Dab)$ and the corresponding semisimple block of $Fil_k(\widehat{\co}^{\mathfrak{p}_n}_{\nu, \bC^n})$.

\InnaD{
Now, fix a non-semisimple block $\mathcal{B}_{\lambda}$ of $\Dab$, and denote by $ Fil_k(\mathcal{B}_{\lambda})$ the corresponding non-semisimple block of $ Fil_k(\Dab)$. We denote by $\mathfrak{B}_{\lambda, n}$ the corresponding block in $\co^{\mathfrak{p}_n}_{\nu, \bC^n}$. The corresponding block of $Fil_k(\widehat{\co}^{\mathfrak{p}_n}_{\nu, \bC^n})$ will then be $\hat{\pi}(Fil_k(\mathfrak{B}_{\lambda, n}))$.

We now check that the contravariant functor $$\widehat{SW}_{\nu, \bC^n}\rvert_{Fil_k(\mathcal{B}_{\lambda})} :  Fil_k(\mathcal{B}_{\lambda}) \rightarrow \hat{\pi}(Fil_k(\mathfrak{B}_{\lambda, n}))$$ is an anti-equivalence of categories when $n \geq k+1$. 

Since $n \geq k+1$, the Serre subcategories $Fil_k(\mathcal{B}_{\lambda})$ and $Ker(\widehat{SW}_{\nu, \bC^n})$ of $\Dab$ have trivial intersection (see Lemma \ref{lem:SW_images_simples}), which means that the restriction of $\widehat{SW}_{\nu, \bC^n}$ to the Serre subcategory $  Fil_k(\mathcal{B}_{\lambda})$ is both faithful and full (the latter follows from Theorem \ref{thrm:SW_almost_equiv}).

It remains to establish that the functor $\widehat{SW}_{\nu, \bC^n}\rvert_{Fil_k(\mathcal{B}_{\lambda})}$ is essentially surjective when 

$n \geq k+1$. This can be done by checking that this functor induces a bijection between the sets of isomorphism classes of indecomposable projective objects in $Fil_k(\mathcal{B}_{\lambda}), \hat{\pi}(Fil_k(\mathfrak{B}_{\lambda, n}))$ respectively (see \cite[Proof of Theorem 7.2.7]{EA} where we use a similar technique). The latter fact follows from the proof of \cite[Theorem 7.2.7]{EA}.


%
Thus $\widehat{SW}_{\nu, \bC^n}: Fil_k(\mathcal{B}_{\lambda}) \rightarrow Fil_k(\hat{\pi}(\mathfrak{B}_{\lambda, n}))$ is an anti-equivalence of categories for 

$n \geq k+1$, and 
$$\widehat{SW}_{\nu, \bC^n}: Fil_k(\Dab) \rightarrow Fil_k(\widehat{\co}^{\mathfrak{p}_n}_{\nu, \bC^n})$$ is an anti-equivalence of categories for $n \geq k+1$, which completes the proof.
}
%
%
 \end{proof}

\section{Schur-Weyl duality for \texorpdfstring{$\underline{Rep}^{ab}(S_{\nu})$ and $\gl_{\infty}$}{Deligne category and the infinite Lie algebra gl}}\label{sec:SW_duality_inf_dim}
\subsection{}
Let $\bC^{\infty}$ be a complex vector space with a countable basis $e_1, e_2, e_3, ...$. Fix $\triv:=e_1$ and $U_{\infty}:=span_{\bC}(e_2, e_3, ...)$. 

\begin{lemma}\label{lem:SW_Gamma_compatible}
 We have a commutative diagram
 $$\xymatrix{&\underline{Rep}^{ab}(S_{\nu})^{op} \ar[rr]_{\widehat{SW}_{\nu, \InnaD{\text{lim}}}}  \ar[rrd]_{\widehat{SW}_{\nu, \bC^{\infty}}} &{} &\varprojlim_{n \geq 1, \text{ restr}} \widehat{\co}^{\mathfrak{p}_{n}}_{\nu, \bC^{n}} \\ &{} &{} &\widehat{\co}^{\mathfrak{p}_{\infty}}_{\nu, \bC^{\infty}} \ar[u]_{\widehat{\Gamma}_{\text{lim}}} }$$
Namely, there is a natural isomorphism $\hat{\eta}: \widehat{\Gamma}_{\text{lim}} \circ \widehat{SW}_{\nu, \bC^{\infty}} \rightarrow \widehat{SW}_{\nu, \InnaD{\text{lim}}}$.
\end{lemma}
\begin{proof}
In order to prove this \InnaD{statement}, we will show that for any $\InnaD{n \geq 1}$, the following diagram is commutative:
$$ \xymatrix{&\underline{Rep}^{ab}(S_{\nu})^{op}  \ar[rr]^-{{SW}_{\nu, \bC^n}} \ar[rrd]^{{SW}_{\nu, \bC^{\infty}}} \ar@/^2pc/[rrrr]^{\widehat{SW}_{\nu, \bC^n}} \ar@/_3pc/[rrrrd]_{\widehat{SW}_{\nu, \bC^{\infty}}} &{} &{\co}^{\mathfrak{p}_{n}}_{\nu, \bC^{n}} \ar[rr]_{\hat{\pi}_n} &{}  &\widehat{\co}^{\mathfrak{p}_{n}}_{\nu, \bC^{n}} \\ &{} &{} &{\co}^{\mathfrak{p}_{\infty}}_{\nu, \bC^{\infty}} \ar[rr]^{\hat{\pi}_{\infty}} \ar[u]_{{\Gamma}_n} &{} &\widehat{\co}^{\mathfrak{p}_{\infty}}_{\nu, \bC^{\infty}} \ar[u]_{\widehat{\Gamma}_n}}$$
\InnaD{which will prove the required statement. The commutativity of this diagram  follows from the existence of a natural isomorphism $ {\Gamma}_{n} \circ {SW}_{\nu, \bC^{\infty}} \stackrel{\sim}{\rightarrow} {SW}_{\nu, \bC^n}$ (due to Proposition \ref{prop:compatab_comp_tens_power}) and a natural isomorphism $\widehat{\Gamma}_n \circ \hat{\pi}_{\infty} \cong \hat{\pi}_n \circ \Gamma_n$ (see proof of Proposition \ref{prop:inv_lim_hat_gl_infty_par_cat}). }
 
\end{proof}

Thus we obtain a commutative diagram
 $$\xymatrix{&{} &{} &\widehat{\co}^{\mathfrak{p}_{n}}_{\nu, \bC^{n}} \\ &\underline{Rep}^{ab}(S_{\nu})^{op} \ar[rr]_{\widehat{SW}_{\nu, \InnaD{\text{lim}}}}  \ar[rru]^{\widehat{SW}_{\nu, \bC^n}} \ar[rrd]_{\widehat{SW}_{\nu, \bC^{\infty}}} &{} &\varprojlim_{n \geq 1, \text{ restr}} \widehat{\co}^{\mathfrak{p}_{n}}_{\nu, \bC^{n}} \ar[u]_{\mathbf{Pr}_n} \\ &{} &{} &\widehat{\co}^{\mathfrak{p}_{\infty}}_{\nu, \bC^{\infty}} \ar[u]_{\widehat{\Gamma}_{\text{lim}}} \ar@/_4.7pc/[uu]_{\widehat{\Gamma}_n} }$$

\begin{theorem}\label{thrm:SW_duality_inf_dim}
 The contravariant functor $\widehat{SW}_{\nu, \bC^{\infty}}:\underline{Rep}^{ab}(S_{\nu}) \rightarrow \widehat{\co}^{\mathfrak{p}_{\infty}}_{\nu, \bC^{\infty}}$ is an anti-equivalence of abelian categories.
\end{theorem}
\begin{proof}
The functor $\widehat{\Gamma}_{\text{lim}}$ is an equivalence of categories (see Lemma \ref{prop:inv_lim_hat_gl_infty_par_cat}), and the functor $\widehat{SW}_{\nu, \InnaD{\text{lim}}}$ is an anti-equivalence of categories (see Theorem \ref{thrm:SW_limit_equivalence}). The commutative diagram above implies that the contravariant functor $\widehat{SW}_{\nu, \bC^{\infty}}$ is an anti-equivalence of categories as well.
\end{proof}

\section{Schur-Weyl functors and duality structures}\label{sec:duality}
\subsection{}
Let $\InnaD{n} \in \bZ_+$. 

Recall the contravariant duality functors $$(\cdot)^{\vee}_n: \left(\co^{\mathfrak{p}_n}_{\nu, \bC^n}\right)^{op} \rightarrow  \co^{\mathfrak{p}_n}_{\nu, \bC^n} $$ discussed in Subsection \ref{ssec:structure_cat_O}. This functor takes \InnaD{polynomial} modules to \InnaD{polynomial} modules, and therefore descends to a duality functor 
$$\widehat{(\cdot)}^{\vee}_n: \left(\widehat{\co}^{\mathfrak{p}_n}_{\nu, \bC^n}\right)^{op} \rightarrow  \widehat{\co}^{\mathfrak{p}_n}_{\nu, \bC^n} $$

\InnaG{Next, the definition of duality functor in $\co^{\mathfrak{p}_n}_{\nu, \bC^n}$} that the duality functors commute with the restriction functors $\InnaD{\mathfrak{Res}}_{n-1, n}$, namely, that for any $n \geq 2$, we have:
$$(\cdot)^{\vee}_{n-1} \circ \InnaD{\mathfrak{Res}}_{n-1, n}^{op} \cong \InnaD{\mathfrak{Res}}_{n-1, n}^{op} \circ (\cdot)^{\vee}_{n} $$

%

This allows us to define duality functors
$$(\cdot)^{\vee}_{\InnaD{\text{lim}}}: \left( \varprojlim_{\InnaD{n \geq 1}, \text{ restr}} {\co}^{\mathfrak{p}_n}_{\nu, \bC^n} \right)^{op} \rightarrow  \varprojlim_{\InnaD{n \geq 1}, \text{ restr}} {\co}^{\mathfrak{p}_n}_{\nu, \bC^n} $$
and $$\widehat{(\cdot)}^{\vee}_{\InnaD{\text{lim}}}: \left( \varprojlim_{\InnaD{n \geq 1}, \text{ restr}} \widehat{\co}^{\mathfrak{p}_n}_{\nu, \bC^n} \right)^{op} \rightarrow  \varprojlim_{\InnaD{n \geq 1}, \text{ restr}} \widehat{\co}^{\mathfrak{p}_n}_{\nu, \bC^n} $$

Under the equivalence \InnaD{$\co^{\mathfrak{p}_{\infty}}_{\nu, \bC^{\infty}} \cong \varprojlim_{\InnaD{n \geq 1}, \text{ restr}} {\co}^{\mathfrak{p}_n}_{\nu, \bC^n}$} established in Subsection \ref{ssec:inv_lim_par_cat_O}, the functor $(\cdot)^{\vee}_{\InnaD{\text{lim}}}$ corresponds to the duality functor
$$(\cdot)^{\vee}_{\infty}: \left(\co^{\mathfrak{p}_{\infty}}_{\nu, \bC^{\infty}}\right)^{op} \rightarrow  \co^{\mathfrak{p}_{\infty}}_{\nu, \bC^{\infty}} $$
discussed in Subsection \ref{ssec:duality_cat_O}. \InnaG{Again, this functor
descends to a contravariant duality functor}
$$\widehat{(\cdot)}^{\vee}_{\infty}: \left(\widehat{\co}^{\mathfrak{p}_{\infty}}_{\nu, \bC^{\infty}}\right)^{op} \rightarrow  \widehat{\co}^{\mathfrak{p}_{\infty}}_{\nu, \bC^{\infty}} $$

As a consequence of Theorem \ref{thrm:SW_almost_equiv}, a connection was established between the notions of duality in the Deligne category $\Dab$ and the duality in the category $\widehat{\co}^{\mathfrak{p}_N}_{\nu, {\bC^N}}$ for $N \in \bZ_+$ (see \cite[Section 7.3]{EA}). The above construction allows us to extend this connection to the case when $N = \infty$. \InnaD{Namely, Theorems \ref{thrm:SW_limit_equivalence} and \ref{thrm:SW_duality_inf_dim}, together with \cite[Section 7.3]{EA}, imply the following statement:}

\begin{proposition}
 Let $N \in \bZ_+ \cup \{ \infty \}$, $\nu \in \bC$. There is an isomorphism of (covariant) functors $$ \widehat{SW}_{\nu, {\bC^N}} \circ ( \cdot )^* \longrightarrow \widehat{( \cdot )}^{\vee}_{N} \circ SW_{\nu, \bC^N}$$
\end{proposition}

\end{document}